\documentclass[11pt]{amsart}
\usepackage{pdfsync}
\usepackage{latexsym}
\usepackage{amsgen,amsmath,amsxtra,amssymb,amsfonts,amsthm}
\usepackage{times}
\usepackage{graphicx}
\usepackage[dvips]{epsfig}
\usepackage{subfigure}
\usepackage[active]{srcltx}
\usepackage{color}

\begin{document}
\newcommand{\dyle}{\displaystyle}
\newcommand{\R}{{\mathbb{R}}}
 \newcommand{\Hi}{{\mathbb H}}
\newcommand{\Ss}{{\mathbb S}}
\newcommand{\N}{{\mathbb N}}
\newcommand{\Rn}{{\mathbb{R}^n}}
\newcommand{\ieq}{\begin{equation}}
\newcommand{\eeq}{\end{equation}}
\newcommand{\ieqa}{\begin{eqnarray}}
\newcommand{\eeqa}{\end{eqnarray}}
\newcommand{\ieqas}{\begin{eqnarray*}}
\newcommand{\eeqas}{\end{eqnarray*}}
\newcommand{\Bo}{\put(260,0){\rule{2mm}{2mm}}\\}
\def\L#1{\label{#1}} \def\R#1{{\rm (\ref{#1})}}


\theoremstyle{plain}
\newtheorem{theorem}{Theorem} [section]
\newtheorem{corollary}[theorem]{Corollary}
\newtheorem{lemma}[theorem]{Lemma}
\newtheorem{proposition}[theorem]{Proposition}
\def\neweq#1{\begin{equation}\label{#1}}
\def\endeq{\end{equation}}
\def\eq#1{(\ref{#1})}

\theoremstyle{definition}
\newtheorem{definition}[theorem]{Definition}
\newtheorem{remark}[theorem]{Remark}

\numberwithin{figure}{section}
\newcommand{\res}{\mathop{\hbox{\vrule height 7pt width .5pt depth
0pt \vrule height .5pt width 6pt depth 0pt}}\nolimits}
\def\at#1{{\bf #1}: } \def\att#1#2{{\bf #1}, {\bf #2}: }
\def\attt#1#2#3{{\bf #1}, {\bf #2}, {\bf #3}: } \def\atttt#1#2#3#4{{\bf #1}, {\bf #2}, {\bf #3},{\bf #4}: }
\def\aug#1#2{\frac{\displaystyle #1}{\displaystyle #2}} \def\figura#1#2{ \begin{figure}[ht] \vspace{#1} \caption{#2}
\end{figure}} \def\B#1{\bibitem{#1}} \def\q{\int_{\Omega^\sharp}}
\def\z{\int_{B_{\bar{\rho}}}\underline{\nu}\nabla (w+K_{c})\cdot
\nabla h} \def\a{\int_{B_{\bar{\rho}}}}
\def\b{\cdot\aug{x}{\|x\|}}
\def\n{\underline{\nu}} \def\d{\int_{B_{r}}}
\def\e{\int_{B_{\rho_{j}}}} \def\LL{{\mathcal L}}
\def\itr{\mathrm{Int}\,}
\def\D{{\mathcal D}}
 \def\tg{\tilde{g}}
\def\A{{\mathcal A}}
\def\S{{\mathcal S}}
\def\H{{\mathcal H}}
\def\M{{\mathcal M}}
\def\T{{\mathcal T}}
\def\U{{\mathcal U}}
\def\N{{\mathcal N}}
\def\I{{\mathcal I}}
\def\F{{\mathcal F}}
\def\J{{\mathcal J}}
\def\E{{\mathcal E}}
\def\F{{\mathcal F}}
\def\G{{\mathcal G}}
\def\HH{{\mathcal H}}
\def\W{{\mathcal W}}
\def\H{\D^{2*}_{X}}
\def\d{d^X_M }
\def\LL{{\mathcal L}}
\def\H{{\mathcal H}}
\def\HH{{\mathcal H}}
\def\itr{\mathrm{Int}\,}
\def\vah{\mbox{var}_\Hi}
\def\vahh{\mbox{var}_\Hi^1}
\def\vax{\mbox{var}_X^1}
\def\va{\mbox{var}}
\def\SS{{\mathcal S}}
 \def\Y{{\mathcal Y}}
\def\length{{l_\Hi}}
\newcommand{\average}{{\mathchoice {\kern1ex\vcenter{\hrule
height.4pt width 6pt depth0pt} \kern-11pt} {\kern1ex\vcenter{\hrule height.4pt width 4.3pt depth0pt} \kern-7pt} {} {} }}
\def\weak{\rightharpoonup}
\def\detu{{\rm det}(D^2u)}
\def\detut{{\rm det}(D^2u(t))}
\def\detvt{{\rm det}(D^2v(t))}
\def\detv{{\rm det}(D^2v)}
\def\uuu{u_xu_yu_{xy}}
\def\uuut{u_x(t)u_y(t)u_{xy}(t)}
\def\uuus{u_x(s)u_y(s)u_{xy}(s)}
\def\uuutn{u_x(t_n)u_y(t_n)u_{xy}(t_n)}
\def\vvv{v_xv_yv_{xy}}
\newcommand{\ave}{\average\int}

\title[Polyharmonic $k-$Hessian equations]{On polyharmonic regularizations of $k-$Hessian equations: Variational methods}

\author[C. Escudero]{Carlos Escudero}
\address{}
\email{}

\date{\today}

\begin{abstract}
This work is devoted to the study of the boundary value problem
\begin{eqnarray}\nonumber
(-1)^\alpha \Delta^\alpha u = (-1)^k S_k[u] + \lambda f, \qquad x &\in& \Omega \subset \mathbb{R}^N, \\ \nonumber
u = \partial_n u = \partial_n^2 u = \cdots = \partial_n^{\alpha-1} u = 0,
\qquad x &\in& \partial \Omega,
\end{eqnarray}
where the $k-$Hessian $S_k[u]$ is the $k^{\mathrm{th}}$ elementary symmetric polynomial of eigenvalues of the Hessian matrix
and the datum $f$ obeys suitable summability properties. We prove the existence of at least two solutions, of which at least one
is isolated, strictly by means of variational methods. We look for the optimal values of $\alpha \in \mathbb{N}$ that allow the construction
of such an existence and multiplicity theory and also investigate how a weaker definition of the nonlinearity permits improving these results.
\end{abstract}
\renewcommand{\thefootnote}{\fnsymbol{footnote}}
\setcounter{footnote}{-1}
\footnote{Supported by projects MTM2013-40846-P and RYC-2011-09025, MINECO, Spain.}
\renewcommand{\thefootnote}{\arabic{footnote}}
\maketitle

\section{Introduction}

The main objective of this paper is studying elliptic equations of the form
\begin{equation}\label{rkhessian}
(-1)^\alpha \Delta^\alpha u = (-1)^k S_k[u] + \lambda f, \qquad x \in \Omega \subset \mathbb{R}^N,
\end{equation}
where $\alpha, \, N, \, k \, \in \mathbb{N}$, $\lambda \in \mathbb{R}$ and $f:\mathbb{R}^N \longrightarrow \mathbb{R}$ fulfills
suitable summability properties (see below). All throughout this work $\Omega$ will denote a bounded and open domain provided
with a smooth boundary $\partial \Omega$. The nonlinearity in~\eqref{rkhessian} is the $k-$Hessian $S_k[u]=\sigma_k(\Lambda)$ where
$$
\sigma_k(\Lambda)= \sum_{i_1<\cdots<i_k} \Lambda_{i_1} \cdots \Lambda_{i_k},
$$
is the $k^{\mathrm{th}}$ elementary symmetric polynomial and $\Lambda=(\Lambda_1,\cdots,\Lambda_n)$ are the eigenvalues of the Hessian matrix $(D^2 u)$.
Equivalently we could say that $S_k[u]$ is the sum of the $k^{\mathrm{th}}$ principal minors of the Hessian matrix.
We employ the notation of~\cite{wang} and denote
$$
S^{ij}_k(D^2 u) = \left. \frac{\partial}{\partial a_{ij}} \sigma_k[\Lambda(A)] \right|_{A=D^2 u},
$$
where $\Lambda(A)$ are the eigenvalues of the $N \times N$ matrix $A$ which entries are $a_{ij}$.
For $k=1$ equation~\eqref{rkhessian} becomes linear. Since we are interested in nonlinear equations
we will always assume $2 \le k \le N$.

The goal of this work is threefold. Our first source of inspiration is the impressive development of theory concerning the $k-$Hessian equations
$$
S_k[u] = f,
$$
and related problems that has taken place during the last
years~\cite{caffarelli3,wang1,wang2,labutin,wang3,trudinger,trudinger1,wang4,wang5,wang6,wang7,wang8,wang9,wang10,wang11,wang12,wang}.
This equation is both a generalization of the Monge-Amp\`ere or $N-$Hessian equation~\cite{caffarelli1,caffarelli2}
$$
\det(D^2 u) = f,
$$
and the Poisson or $1-$Hessian equation~\cite{gilbarg}\footnote{Properly speaking, the $1-$Hessian would not show the minus in front of the Laplacian,
although this correction is of course completely irrelevant for the differential problem.}
$$
-\Delta u = f.
$$
In order to ensure the ellipticity of the nonlinear $k-$Hessian equations one should look for solutions $u$ such that
$$
\Lambda(D^2 u) \in \bar{\gimel}_k,
$$
where $\gimel_k$ is an open symmetric convex cone in $\mathbb{R}^N$ defined in the following way
$$
\gimel_k = \{(\Lambda_1, \cdots, \Lambda_N) \in \mathbb{R}^N | \sigma_j >0 \,\, \forall \, j = 1, \cdots, k\};
$$
all functions fulfilling this property are called $k-$admissible~\cite{wang}.
The existence theory for the Dirichlet problem associated to the $k-$Hessian equations requires a geometric assumption
on the boundary of the domain $\partial \Omega$ in which such a problem is posed. In particular
$$
\sigma_{k-1}(\kappa_1, \cdots, \kappa_{N-1}) \ge C_0 > 0,
$$
on $\Omega$ for some positive constant $C_0$, where $\kappa_i$, $i=1, \cdots, N-1$, are the principal curvatures of $\Omega$
with respect to its inner normal~\cite{wang}. As in this reference, we will denote all domains fulfilling this property as $(k-1)-$convex.
On one hand, turning a $k-$Hessian equation into one of its polyharmonic regularizations~\eqref{rkhessian}
with $\alpha \ge 2$ (an assumption that will hold all throughout this text) means turning a fully nonlinear problem into a semilinear, and thus simpler, one.
On the other hand, however, considering the existence theory for~\eqref{rkhessian} allows us to drop two assumptions that were necessary
in the fully nonlinear case: our solutions do not have to be $k-$admissible and the boundary of our domains does not have to be $(k-1)-$convex.
So this type of problem yields a different viewpoint on an interesting nonlinearity.

Our second source of inspiration are polyharmonic problems by themselves. Despite their relevance in different applications and intrinsic
mathematical interest, they have being much less studied than their harmonic counterparts. This could be perhaps due to the usual absence
of maximum principles in the polyharmonic case, while such principles played in fact a crucial role in the development of the theory for
second order problems. Although our present knowledge of higher order boundary value problems cannot be compared to the corresponding one for
second order boundary value problems, it has nevertheless substantially grown over the last years~\cite{GGS}. For instance, boundary value
problems for the biharmonic operator have already been considered with different nonlinearities~\cite{AGGM,BG,CEGM,DDGM,DFG,FG,FGK,moradifam}.
However, to the best of our knowledge, a Hessian nonlinearity was considered for this operator in~\cite{n5} for the first time. Despite the
novelty of this sort of problem, it is rather natural to consider biharmonic, or more in general polyharmonic, equations provided with
nonlinear functions of the second derivatives of the solution. The most natural candidates for these nonlinearities are the $k-$Hessians for the
following reason: the Hessian matrix, which entries are all possible second derivatives of the solution, possesses exactly $N$ tensorial invariants,
the $N$ different $k-$Hessians. Therefore one of our present goals is to continue and to generalize our studies on this type of
problems~\cite{n1,n2,n3,n4,n5,n6}, that so naturally appear in the theory of higher order partial differential equations. A related question that could
be of independent interest is the presence of fractional rather than polyharmonic operators. A possible starting point to approach this kind of problem
could be works such as~\cite{FeVe}, at least when radially symmetric solutions are considered~\cite{n6}.

Our last motivation is the connection of these equations with theoretical condensed matter physics and the renormalization group~\cite{escudero,escudero2}.
In particular this sort of equations has been proposed to describe the growth of some semiconductor structures by means of epitaxial methods.
In this context the solution to the partial differential equation describes the height of the grown crystal and the domain where the equation is defined
the substrate over which it is grown. This more phenomenological framework
also opens the possibility of studying these equations in the field of non-equilibrium phase transitions, and in particular within the general theory
of the propagation of stable phases over unstable ones. In this case the solution describes the front separating both. In physics there exists an interest
in studying the dynamics and morphology of such fronts, a problem commonly approached with multi-scale methods like the renormalization group. An interesting
fact is how the nonlinearities we consider transform under the renormalization group, what in turn is related to intriguing underlying physics, although this
question is not free from technicalities.
Despite the potential interest of our results in this field, in the present work we will limit ourselves to the development of mathematical theory
and leave any application for the future.

For the time being we will restrict ourselves to the study of the polyharmonic boundary value problem
\begin{eqnarray}\label{dirichlet}
(-1)^\alpha \Delta^\alpha u = (-1)^k S_k[u] + \lambda f, \qquad x &\in& \Omega \subset \mathbb{R}^N, \\ \nonumber
u = \partial_n u = \partial_n^2 u = \cdots = \partial_n^{\alpha-1} u = 0,
\qquad x &\in& \partial \Omega,
\end{eqnarray}
which we refer to as the Dirichlet problem for partial differential equation~\eqref{rkhessian}.
The main theoretical tool that we employ in building the existence theory for this problem is the calculus of variations.
Despite the interest of studying different boundary value problems, such as
\begin{eqnarray}\label{navier}
(-1)^\alpha \Delta^\alpha u = S_k[u] + \lambda f, \qquad x \in \Omega \subset \mathbb{R}^N, \\ \nonumber
u = \Delta u = \Delta^2 u = \cdots = \Delta^{\alpha-1} u = 0,
\qquad x \in \partial \Omega,
\end{eqnarray}
which we refer to as the Navier problem for partial differential equation~\eqref{rkhessian}, we are not going to do so
in the present work. The reason is that variational methods cannot be applied in this case and different techniques,
like fixed methods, are needed~\cite{n5}. Herein we will limit ourselves to the variational setting.

Now we list the main results of this paper. The definitions of $p^*$ and $h^1_r(\Omega)$ come in the next section.

\begin{theorem}\label{main}
Let
\begin{itemize}
\item[(i)] $\alpha= \left\lceil 2+ \frac{k-2}{2k}N \right\rceil$, $f \in L^1(\Omega)$ for $N/2 < k \le N$,
\item[(ii)] $\alpha= \left\lceil \frac{Nk-N+4k}{2k+2} \right\rceil$, $f \in L^{p^*}(\Omega)$ for $2 \le k < N/2$, and
\item[(iii)] $\alpha= N/2$, $f \in h^1_r(\Omega)$ for $k = N/2$.
\end{itemize}
Then there exist a $\lambda_0 > 0$ such that for $0 \le |\lambda| < \lambda_0$ and $2 \le k \le N$ problem~\eqref{dirichlet}
has at least two solutions. Moreover, these solutions differ in $W^{\alpha,2}_0(\Omega)$ norm and one of them is unique in the ball
$$
\mathcal{B}=\left\{u \in W^{\alpha,2}_0(\Omega) : 0 \le \|u\|_{W^{\alpha,2}_0(\Omega)} \le \tau \right\},
$$
for some $\tau>0$.
\end{theorem}

\begin{proof}
This result is a direct consequence of theorem~\ref{existmul}, theorem~\ref{existuni}, corollary~\ref{unimul}, remark~\ref{lambda0}, theorem~\ref{existwc}
and theorem~\ref{existwc2}.
\end{proof}

Now we will see how a weaker definition of the nonlinearity allows to build our existence and multiplicity theory for lower values of $\alpha$.
Consider the boundary value problem
\begin{eqnarray}\label{dirichletw}
(-1)^\alpha \Delta^\alpha u = (-1)^k \mathcal{S}_k[u] + \lambda f, \qquad x &\in& \Omega \subset \mathbb{R}^N, \\ \nonumber
u = \partial_n u = \partial_n^2 u = \cdots = \partial_n^{\alpha-1} u = 0,
\qquad x &\in& \partial \Omega,
\end{eqnarray}
where $\mathcal{S}_k[u]$ is the weak $k-$Hessian defined by equation~\eqref{weakdirichlet}.

\begin{theorem}\label{mainw}
Let
\begin{itemize}
\item[(i)] $f \in L^1(\Omega)$ for $N/2 < k \le N$,
\item[(ii)] $f \in L^{p^*}(\Omega)$ for $2 \le k < N/2$, and
\item[(iii)] $f \in h^1_r(\Omega)$ for $k = N/2$,
\end{itemize}
and $\alpha= \left\lceil \frac{Nk-N+4k}{2k+2} \right\rceil$ in all cases.
Then there exist a $\lambda_0 > 0$ such that for $0 \le |\lambda| < \lambda_0$ and $2 \le k \le N$ problem~\eqref{dirichletw}
has at least two solutions. Moreover, these solutions differ in $W^{\alpha,2}_0(\Omega)$ norm and one of them is unique in the ball
$$
\mathcal{B}=\left\{u \in W^{\alpha,2}_0(\Omega) : 0 \le \|u\|_{W^{\alpha,2}_0(\Omega)} \le \tau \right\},
$$
for some $\tau>0$.
\end{theorem}

\begin{proof}
This result is a direct consequence of the auxiliary results that lead to theorem~\ref{main} together with theorem~\ref{existwkh}.
\end{proof}

The remainder of the paper is devoted to prove all the auxiliary results needed in the proof of our main theorems. These are placed in the following sections:
in section~\ref{variational} we develop the variational formulation of the problem and prove existence and multiplicity of solutions to the problem at hand
in the range $N/2 < k \le N$. The proof makes use of both the {\it mountain pass} and the Arzel\`a-Ascoli theorems.
In the range $2 \le k \le N/2$ we prove existence
and multiplicity of solutions substituting the Arzel\`a-Ascoli theorem by suitable weak continuity results; this is recorded in section~\ref{weakcontsec}.
Moreover, we show in this section how substituting the $k-$Hessian by a weak $k-$Hessian (see theorem~\ref{mainw} above)
it is possible to build our existence and multiplicity theory
for a ``uniform'' value of $\alpha$. The proof makes use again of the weak continuity properties of the nonlinearity; this suggests that it is indeed this approach
the most natural for the type of problem at hand (see theorem~\ref{weakcont2} too).
It also seems that the weak formulation of the problem, i.~e.~\eqref{dirichletw}, since it is the one
that allows the use of a more uniform value of $\alpha$, is the most natural from a variational viewpoint.
In section~\ref{further} we collect some other results that complement the previous developments.
In particular, we investigate what happens if we make a uniform hypothesis on the summability of the datum $f$.
This in turn allows to build the existence and multiplicity theory for a uniform value of $\alpha$, but the result is less optimal than
the corresponding one of theorem~\ref{mainw} and even that of theorem~\ref{main}.

Before introducing the technical results let us mention a couple of remarks regarding notation. For the partial derivatives of the solution
we will indistinctively use $\partial_{x_i} u$, $u_{x_i}$, and $u_i$ for $1 \le i \le N$. We will denote constants in different cases and sections with
the same letters but this does not mean that they share the same value. In general, the numerical value of a constant may change from line to line, and
we will only use different letters for different constants when we appreciate some risk of confusion.

\section{Variational approach}\label{variational}
\subsection{Variational and functional settings}\label{variationalset}

As already said in the previous section, our aim is studying boundary value problem~\eqref{dirichlet} by means of variational
methods. The first step is finding a suitable functional for this purpose.
We look for a functional
$$
J[u]: W^{\alpha,2}_0(\Omega) \longrightarrow \mathbb{R},
$$
which is well defined in the functional space dictated by the linear term in the equation.

\begin{lemma}\label{definition}
The functional
\begin{equation}\label{func0}
J[u]= \int \left[ -\lambda f u + \frac{1}{2} \left| \Delta ^{\lfloor \alpha/2 \rfloor}
\nabla^{2\left(\alpha/2-\lfloor \alpha/2 \rfloor \right)} u \right|^2 - \frac{(-1)^k}{k+1} u S_k[u]\right] dx
\end{equation}
is well defined in $W^{\alpha,2}_0(\Omega)$ in the following cases:
\begin{itemize}
\item[(a)] If $\alpha= \left\lceil 2+ \frac{k-2}{2k}N \right\rceil$, $f \in L^1(\Omega)$ and $N/2 < k \le N$.
\item[(b)] If $\alpha= \left\lceil \frac{Nk-N+4k}{2k+2} \right\rceil$, $f \in L^{p^*}(\Omega)$ for $p^*:=N(k+1)/[k(N+2)]$ and $2 \le k < N/2$.
\item[(c)] If $\alpha= N/2+1$, $f \in L^1(\Omega)$ and $k = N/2$.
\end{itemize}
\end{lemma}

\begin{proof}
{\it Step 1. Case $N/2 < k \le N$.}

It is clear that if $u \in W^{2,k}_0(\Omega)$ then $S_k[u] \in L^1(\Omega)$ and also by Sobolev embedding we know that $u \in L^\infty(\Omega)$.
In this case we can invoke H\"{o}lder inequality to obtain the estimates
\begin{eqnarray}\nonumber
\int \left| u \, S_k[u] \right| dx &\le& \|S_k[u]\|_1 \, \|u\|_\infty, \\ \nonumber
\int \left| u \,f \right| dx &\le& \|f\|_1 \, \|u\|_\infty.
\end{eqnarray}
Both right hand sides are well defined as long as $W^{\alpha,2}_0(\Omega) \hookrightarrow W^{2,k}_0(\Omega)$. One can easily check that the exponent $\alpha$ in
the statement is the optimal one for this embedding to hold.

{\it Step 2. Case $2 \le k < N/2$.}

In this case the Sobolev embedding $W^{2,k}_0(\Omega) \hookrightarrow L^\infty(\Omega)$ no longer holds.
On the other hand we may invoke H\"{o}lder inequality again to get the estimate
$$
\int \left| u \, S_k[u] \right| dx \le \|S_k[u]\|_p \, \|u\|_q,
$$
for $p$ and $q$ such that $1/p + 1/q =1$. Now we look for the optimal $\alpha$ such that the string of embeddings
$W^{\alpha,2}_0(\Omega) \hookrightarrow W^{2,kp}_0(\Omega) \hookrightarrow L^q(\Omega)$ is fulfilled.
This is so for the choice $p=p^*:=N(k+1)/[k(N+2)]$, $q=q^*:=N(k+1)/(N-2k)$ and $\alpha$ as in the statement.
Once more, H\"{o}lder inequality leads to
$$
\int \left| u \, f \right| dx \le \|f\|_{p^*} \, \|u\|_{q^*}.
$$
The last right hand side is well defined assuming the hypotheses in the statement.

{\it Step 3. Case $k = N/2$.}

This borderline case differs from the first one in the fact that the Sobolev embedding $W^{N/2,2}_0(\Omega) \hookrightarrow W^{2,N/2}_0(\Omega)$ holds,
but the embedding $W^{2,N/2}_0(\Omega) \hookrightarrow L^{\infty}(\Omega)$ does not. Consequently we need to choose $\alpha=N/2+1$ in order to
have $u \in L^\infty(\Omega)$.

{\it Step 4.}

Noting that the quadratic term in~\eqref{func0} is always well defined by its very nature concludes the proof.
\end{proof}

The following series of remarks is in order.

\begin{remark}
The value of $\alpha$ satisfies the sharp bound $\alpha \ge 2$.
\end{remark}

\begin{remark}
Functional~\eqref{func0} can also be written as
\begin{equation}\nonumber
J[u]= \left\{ \begin{array}{lll} \int \left( -\lambda f u + \frac{1}{2} \left| \Delta ^{\alpha/2}
u \right|^2 - \frac{(-1)^k}{k+1} u S_k[u]\right) dx,
\qquad \mathrm{if} \quad \alpha \quad \mathrm{is} \,\,\, \mathrm{even}, \\ \\
\int \left( -\lambda f u + \frac{1}{2} \left| \Delta ^{ (\alpha-1)/2}
\nabla u \right|^2 - \frac{(-1)^k}{k+1} u S_k[u]\right) dx,
\qquad \mathrm{if} \quad \alpha \quad \mathrm{is} \,\,\, \mathrm{odd}.
\end{array} \right.
\end{equation}
\end{remark}

\begin{remark}
Case (a) takes place for any $N \ge 2$, case (b) takes place for any $N \ge 5$, and case (c) takes place for any even $N \ge 4$.
\end{remark}

\begin{remark}\label{endpointrem}
Case (c) is not a consequence of case (a) (resp. case (b)) when $k$ approaches its lower (resp. upper) limit.
\end{remark}

\begin{remark}
Exponent $p^*$ satisfies the sharp bounds $1 < p^* <3/2$.
\end{remark}

\begin{remark}
Lemma~\ref{definition} exhausts the possibilities for $k$ and $N$.
\end{remark}

\begin{remark}
In case (a) we can rewrite $\alpha$ as follows:
\begin{itemize}
\item $\alpha=(N+2)/2$ if $N$ is even.
\item If $N$ is odd then
\begin{itemize}
\item $\alpha= (N+1)/2$ if $k \le \lfloor 2N/3 \rfloor$,
\item $\alpha= (N+3)/2$ if $k \ge \lfloor 2N/3 +1 \rfloor$.
\end{itemize}
\end{itemize}
Equivalently this can be written in the following way:
\begin{itemize}
\item $\alpha= \lfloor N/2 + 1 \rfloor$ if $k \le \lfloor 2N/3 \rfloor$,
\item $\alpha= \lceil N/2 + 1 \rceil$ if $k \ge \lfloor 2N/3 +1 \rfloor$,
\end{itemize}
\end{remark}

\begin{remark}
A more general result is obtained in both cases (a) and (c) if we let $f \in W^{-\alpha,2}(\Omega) \supset L^1(\Omega)$ and interpret
$$
\int f \, u \, dx \equiv \left\langle f,u \right\rangle.
$$
\end{remark}

\begin{remark}
The values of $\alpha$ in lemma~\ref{definition} are optimal in the sense of Sobolev embeddings of Sobolev spaces into
other Sobolev or Lebesgue spaces. If other functional spaces are invoked, then improvements may be possible, see lemma~\ref{localhardy} below.
\end{remark}

As it has already been outlined in the previous remark, the marginal character of case (c) in lemma~\ref{definition} suggests that improvements
are possible. To put this intuitive observation on a precise ground we need to introduce
the Hardy space in $\mathbb{R}^N$~\cite{stein} and its local counterparts~\cite{der}.

\begin{definition}\label{hardyspace} Let $\Phi \in \mathcal{S}(\mathbb{R}^N)$ be a function such that $\int_{\mathbb{R}^N} \Phi \, dx =1$.
Define $\Phi_s := s^{-N} \Phi(x/s)$ for $s>0$. A locally integrable function $f$ is said to be in $\mathcal{H}^1(\mathbb{R}^N)$ if
the maximal function
$$
\mathcal{M}f(x):= \sup_{s>0} \left| \Phi_s \ast f(x) \right|
$$
belongs to $L^1(\mathbb{R}^N)$. We define the norm $\|f\|_{\mathcal{H}^1(\mathbb{R}^N)}=\|\mathcal{M} f\|_1$.
\end{definition}

\begin{remark}
There are several equivalent definitions of this space, see~\cite{steinb}.
\end{remark}

\begin{definition}
\label{localhardyspace} Let $\phi \in \mathcal{S}(\mathbb{R}^N)$ be a function such that $\int_{\mathbb{R}^N} \phi \, dx =1$.
Define $\phi_s := s^{-N} \phi(x/s)$ for $s>0$. A locally integrable function $f$ is said to be in $h^1(\mathbb{R}^N)$ if
the maximal function
$$
m f(x):= \sup_{0<s<1} \left| \phi_s \ast f(x) \right|
$$
belongs to $L^1(\mathbb{R}^N)$. We define the norm $\|f\|_{h^1(\mathbb{R}^N)}=\|m f\|_1$.
\end{definition}

\begin{definition}
We will denote as $h_r^1(\Omega)$ the space of locally integrable functions which are the restrictions to $\Omega$ of elements of $h^1(\mathbb{R}^N)$. This space is equipped with the quotient norm
$$
\|f\|_{h_r^1(\Omega)}= \inf_F \|F\|_{h^1(\mathbb{R}^N)},
$$
where the infimum is taken over all the functions $F \in h^1(\mathbb{R}^N)$ such that $\left. F \right|_\Omega = f$.
\end{definition}

\begin{remark}
For alternative characterizations of this space see~\cite{goldberg}.
\end{remark}

\begin{definition}
The space $h^1_z(\bar{\Omega})$ is defined to be the subspace of $h^1(\mathbb{R}^N)$ consisting of those elements which are supported
on $\bar{\Omega}$:
$$
h^1_z(\bar{\Omega}):=\{ f \in h^1(\mathbb{R}^N): f=0 \quad \mathrm{on} \quad \mathbb{R}^N \setminus \bar{\Omega} \}.
$$
We define the norm $\|f\|_{h^1_z(\bar{\Omega})}=\|f\|_{h^1(\mathbb{R}^N)}$.
\end{definition}

\begin{lemma}\label{hardyz}
For any $u \in W^{2,N/2}_0(\Omega)$ we have that $S_{N/2}[u] \in h_z^1(\bar{\Omega})$.
\end{lemma}

\begin{proof}
As $W^{2,N/2}_0(\Omega)$ is the closure of $C_0^\infty(\Omega)$ in $W^{2,N/2}(\Omega)$, we can extend $u \in W^{2,N/2}_0(\Omega)$ by zero
to find $\tilde{u} \in W^{2,N/2}(\mathbb{R}^N)$, where $\tilde{u}$ is the zero extension of $u$.
It is clear that $S_{N/2}[\tilde{u}]$ is well defined in $L^1(\mathbb{R}^N)$.
Now we claim
$$
\int_{\mathbb{R}^N} S_{N/2}[v] \, dx=0 \, \, \forall \, v \in C_0^\infty(\mathbb{R}^N).
$$
This follows from the divergence form of $S_k[v]=\frac{1}{k}\sum_{i,j} \partial_{x_i}(v_{x_j} S_k^{ij}[v])$ for all $v \in C_0^\infty(\Omega)$~\cite{wang}.
Then we may invoke the results in~\cite{grafakos1}
(see also~\cite{grafakos,coifman}) to get
$$
S_{N/2}[\tilde{u}] \in \mathcal{H}^1(\mathbb{R}^N).
$$
From definitions~\ref{hardyspace} and~\ref{localhardyspace} it is clear that $\mathcal{H}^1(\mathbb{R}^N) \subset h^1(\mathbb{R}^N)$, and since
$S_{N/2}[\tilde{u}]$ is compactly supported in $\bar{\Omega}$ the desired conclusion follows.
\end{proof}

Now we introduce the space of functions of bounded mean oscillation~\cite{steinb} and some local counterparts~\cite{der}.

\begin{definition}
A locally integrable function $f$ is said to be in $\mathrm{BMO}(\mathbb{R}^N)$ if the seminorm (or norm in the quotient space
of locally integrable functions modulo additive constants)
$$
\|f\|_{\mathrm{BMO}(\mathbb{R}^N)}:= \sup_Q \frac{1}{|Q|}\int_Q |f(x)-f_Q| \, dx,
$$
where $|Q|$ is the Lebesgue measure of $Q$, $f_Q=\frac{1}{|Q|}\int_Q f(x) \, dx$ and the supremum is taken
over the set of all cubes $Q \subset \mathbb{R}^N$, is finite.
\end{definition}

\begin{definition}\label{bmodef}
A locally integrable function $f$ is said to be in $\mathrm{bmo}(\mathbb{R}^N)$ if the norm
$$
\|f\|_{\mathrm{bmo}(\mathbb{R}^N)}:= \sup_{|Q|<1} \frac{1}{|Q|}\int_Q |f(x)-f_Q| \, dx + \sup_{|Q| \ge 1} \frac{1}{|Q|}\int_Q |f(x)| \, dx
$$
is finite. Here the suprema are taken over all cubes $Q \subset \mathbb{R}^N$ with sides parallel to the axes.
\end{definition}

\begin{definition}
A locally integrable function $f$ is said to be in $\mathrm{bmo}_r(\Omega)$ if the norm
$$
\|f\|_{\mathrm{bmo}_r(\Omega)}:= \sup_{|Q|<1} \frac{1}{|Q|}\int_Q |f(x)-f_Q| \, dx + \sup_{|Q| \ge 1} \frac{1}{|Q|}\int_Q |f(x)| \, dx
$$
is finite. Here the suprema are taken over all cubes $Q \subset \Omega$.
\end{definition}

\begin{definition}\label{bmozdef}
The space $\mathrm{bmo}_z(\bar{\Omega})$ is defined to be the subspace of $\mathrm{bmo}(\mathbb{R}^N)$ consisting of those elements which
are supported in $\bar{\Omega}$, with
$$\|f\|_{\mathrm{bmo}_z(\bar{\Omega})}=\|f\|_{\mathrm{bmo}(\mathbb{R}^N)}.$$
\end{definition}

Now we are ready to proof an extension of lemma~\ref{definition}.

\begin{lemma}\label{localhardy}
If $k=N/2$ and $f \in h_r^1(\Omega)$ then functional~\eqref{func0} is well defined in $W_0^{N/2,2}(\Omega)$.
\end{lemma}

\begin{proof}
The embedding $W_0^{N/2,2}(\Omega) \hookrightarrow W_0^{2,N/2}(\Omega)$ and lemma~\ref{hardyz} ensure that $S_{N/2}[u] \in h^1_z(\bar{\Omega})$.
Moreover, we can argue as in the proof of lemma~\ref{hardyz} to show that the zero extension of $u$, $\tilde{u} \in W^{N/2,2}(\mathbb{R}^N)$,
whenever $u \in W_0^{N/2,2}(\Omega)$. Therefore by Sobolev embedding we
get $\tilde{u} \in \mathrm{BMO}(\mathbb{R}^N) \cap L^2(\mathbb{R}^N)$. From here, after recalling definitions~\ref{bmodef} and~\ref{bmozdef},
it is clear that $\tilde{u} \in \mathrm{bmo}_z(\bar{\Omega})$.
Now from~\cite{der} we know that $\mathrm{bmo}_z(\bar{\Omega}) \subset \mathrm{bmo}_r(\Omega)$ and that the following duality relations hold:
$$
[h^1_r(\Omega)]^* = \mathrm{bmo}_z(\bar{\Omega}), \qquad [h^1_z(\bar{\Omega})]^* = \mathrm{bmo}_r(\Omega).
$$
Therefore the proof concludes with the following inequalities obtained by duality:
\begin{eqnarray}\nonumber
\int u \, S_k[u] \, dx &\le& \|S_k[u]\|_{h^1_z(\bar{\Omega})} \, \|u\|_{\text{bmo}_r(\Omega)}, \\ \nonumber
\int u \,f \, dx &\le& \|f\|_{h^1_r(\Omega)} \, \|u\|_{\text{bmo}_z(\bar{\Omega})}.
\end{eqnarray}
\end{proof}

\begin{remark}
Note that strictly speaking lemma~\ref{localhardy} is not an improvement of case (c) in lemma~\ref{definition},
as we are asking for a higher regular $f$.
\end{remark}

\begin{remark}
Considered as subspaces of $L^1(\Omega)$, we have $h^1_z(\bar{\Omega}) \subset h^1_r(\Omega)$, where the inclusion is strict~\cite{der}, so
the statement of lemma~\ref{localhardy} would remain true if we substituted the assumption $f \in h^1_r(\Omega)$ by $f \in h^1_z(\bar{\Omega})$,
but the result would be less general.
\end{remark}

Now we have all the ingredients to prove the following

\begin{proposition}\label{eulag}
Let
\begin{itemize}
\item[(i)] $\alpha= \left\lceil 2+ \frac{k-2}{2k}N \right\rceil$, $f \in L^1(\Omega)$ for $N/2 < k \le N$,
\item[(ii)] $\alpha= \left\lceil \frac{Nk-N+4k}{2k+2} \right\rceil$, $f \in L^{p^*}(\Omega)$ for $2 \le k < N/2$, and
\item[(iii)] $\alpha= N/2$, $f \in h^1_r(\Omega)$ for $k = N/2$.
\end{itemize}
Then the critical points of functional~\eqref{func0} are solutions to boundary value problem~\eqref{dirichlet}.
\end{proposition}

\begin{proof}
By virtue of lemmas~\ref{definition} and~\ref{localhardy} functional~\eqref{func0} is well defined in all three cases.
Now let $v,w \in C^\infty_0(\Omega)$, $t \in \mathbb{R}_+$ and consider $J[v + t w]$
which is well defined as a simple corollary of the previous affirmation.
We find
\begin{eqnarray}\nonumber
& & \left. \frac{d}{dt} J[v + t w] \right|_{t=0}
\\ \nonumber
&=& \int \left\{ -\lambda f w + \left[ \Delta^{\lfloor \alpha/2 \rfloor}
\nabla^{2\left(\alpha/2-\lfloor \alpha/2 \rfloor \right)} v \right]
\left[ \Delta ^{\lfloor \alpha/2 \rfloor}
\nabla^{2\left(\alpha/2-\lfloor \alpha/2 \rfloor \right)} w \right] \right.
\\ \nonumber & & \left.
- \frac{(-1)^k}{k+1} w S_k[v] - \frac{(-1)^k}{k+1} v \left. \frac{d}{dt} S_k[v+ t w] \right|_{t=0} \right\} dx
\\ \nonumber
&=& \int \left\{ -\lambda f w + (-1)^\alpha \Delta^\alpha v \, w - \frac{(-1)^k}{k+1} w S_k[v] \right.
\\ \nonumber & &
- \frac{(-1)^k}{k+1} v
\sum_{i_1<\cdots<i_k} \partial_{D^2 v} \Lambda_{i_1}(D^2 v):D^2 w \cdots \Lambda_{i_k}(D^2 v)
\\ \nonumber & &
- \frac{(-1)^k}{k+1} v
\sum_{j=2}^{k-1} \sum_{i_1<\cdots<i_k} \Lambda_{i_1}(D^2 v) \cdots \partial_{D^2 v} \Lambda_{i_j}(D^2 v) :D^2 w \cdots \Lambda_{i_k}(D^2 v)
\\ \nonumber & & \left.
- \frac{(-1)^k}{k+1} v
\sum_{i_1<\cdots<i_k} \Lambda_{i_1}(D^2 v) \cdots \partial_{D^2 v} \Lambda_{i_k}(D^2 v) :D^2 w \right\} dx
\\ \nonumber
&=& \int \left\{ -\lambda f w + (-1)^\alpha \Delta^\alpha v \, w - \frac{(-1)^k}{k+1} w S_k[v] \right.
\\ \nonumber & & \left.
- \frac{(-1)^k}{k+1} \, v \,\,
\partial_{D^2 v} S_k[v] :D^2 w \right\} dx
\\ \nonumber
&=& \int \left\{ -\lambda f + (-1)^\alpha \Delta^\alpha v - (-1)^k S_k[v] \right\} w \, dx,
\end{eqnarray}
where $\partial_{D^2 v}(\cdot)$ is the matrix which entries are $\partial_{v_{ij}}(\cdot)$,
after repeatedly integrating by parts and using the properties of $S^{ij}_k(D^2 u)$.
By a density argument we can take $v,w \in W^{\alpha,2}_0(\Omega)$ to conclude.
\end{proof}

\begin{remark}
Note that the values of $\alpha$ present in the hypotheses of proposition~\ref{eulag} fulfill
\begin{equation}\label{2lalfa}
\alpha= \left\{ \begin{array}{c}
\left\lceil 2+ \frac{k-2}{2k}N \right\rceil \quad \text{for} \quad N/2 \le k \le N \\ \\
\left\lceil \frac{Nk-N+4k}{2k+2} \right\rceil \quad \text{for} \quad 2 \le k \le N/2
\end{array} \right. .
\end{equation}
In particular, the endpoint value $k=N/2$ coincides in both cases and also with the marginal case (iii), contrary
to what happened in lemma~\ref{definition}, see remark~\ref{endpointrem}. Also, note that this is the only
value of $\alpha$ for which both lines of~\eqref{2lalfa} could coincide.
\end{remark}

\begin{remark}
It is an easy exercise to show that the critical points of functional~\eqref{func0} are not solutions to boundary value problem~\eqref{navier},
see~\cite{n5} for a particular example of this fact.
\end{remark}

\subsection{Geometry of $J[u]$}
\label{geometryj}

This section is devoted to prove that the geometry of functional $J[u]$ corresponds to the {\it mountain pass} one~\cite{ar}.
This will subsequently allow us to prove the existence of at least two solutions to boundary value problem~\eqref{dirichlet}.

\begin{proposition}\label{downbound}
Functional~\eqref{func0} admits the following lower radial estimate in the Sobolev space $W_0^{\alpha,2}(\Omega)$
\begin{eqnarray}\nonumber
G[u] &=& \frac{1}{2} \, \| \Delta ^{\lfloor \alpha/2 \rfloor}\nabla^{2\left(\alpha/2-\lfloor \alpha/2 \rfloor \right)} u \|_2^2
- C_1 \, \|\Delta ^{\lfloor \alpha/2 \rfloor}\nabla^{2\left(\alpha/2-\lfloor \alpha/2 \rfloor \right)} u \|_2 \\ \nonumber & &
- C_2 \, \| \Delta ^{\lfloor \alpha/2 \rfloor}\nabla^{2\left(\alpha/2-\lfloor \alpha/2 \rfloor \right)} u \|_2^{k+1},
\end{eqnarray}
i.~e., $J[u] \ge G[u] \,\, \forall \, u \in W_0^{\alpha,2}(\Omega)$ for suitable positive constants $C_1$ and $C_2$.
\end{proposition}

\begin{proof}

{\it Step 1. Case $N/2 < k \le N$.}

We have the string of inequalities
\begin{eqnarray} \nonumber
J[u] &\ge& \frac{1}{2} \int \left| \Delta ^{\lfloor \alpha/2 \rfloor}
\nabla^{2\left(\alpha/2-\lfloor \alpha/2 \rfloor \right)} u \right|^2 dx
\\ \nonumber
& & - \left|\lambda\right| \|f\|_1 \|u\|_\infty - \frac{1}{k+1} \|u\|_\infty \|S_k[u]\|_1
\\ \nonumber
&\ge& \frac{1}{2} \int \left| \Delta ^{\lfloor \alpha/2 \rfloor}
\nabla^{2\left(\alpha/2-\lfloor \alpha/2 \rfloor \right)} u \right|^2 dx
\\ \nonumber
& & - C_1 \left|\lambda\right| \|f\|_1 \left( \int \left| \Delta ^{\lfloor \alpha/2 \rfloor}
\nabla^{2\left(\alpha/2-\lfloor \alpha/2 \rfloor \right)} u \right|^2 dx \right)^{1/2}
\\ \nonumber
& & - C_2 \left( \int \left| \Delta ^{\lfloor \alpha/2 \rfloor}
\nabla^{2\left(\alpha/2-\lfloor \alpha/2 \rfloor \right)} u \right|^2 dx\right)^{(k+1)/2},
\end{eqnarray}
where we have used two H\"{o}lder inequalities in the first step and two Sobolev embeddings in the second.
The validity of all of them can be checked in the proof of lemma~\ref{definition}.

{\it Step 2. Case $2 \le k < N/2$.}

The corresponding calculation in this case yields
\begin{eqnarray} \nonumber
J[u] &\ge& \frac{1}{2} \int \left| \Delta ^{\lfloor \alpha/2 \rfloor}
\nabla^{2\left(\alpha/2-\lfloor \alpha/2 \rfloor \right)} u \right|^2 dx
\\ \nonumber
& & - \left|\lambda\right| \|f\|_{p^*} \|u\|_{q^*} - \frac{1}{k+1} \|u\|_{q^*} \|S_k[u]\|_{p^*}
\\ \nonumber
&\ge& \frac{1}{2} \int \left| \Delta ^{\lfloor \alpha/2 \rfloor}
\nabla^{2\left(\alpha/2-\lfloor \alpha/2 \rfloor \right)} u \right|^2 dx
\\ \nonumber
& & - C_1 \left|\lambda\right| \|f\|_{p^*} \left( \int \left| \Delta ^{\lfloor \alpha/2 \rfloor}
\nabla^{2\left(\alpha/2-\lfloor \alpha/2 \rfloor \right)} u \right|^2 dx \right)^{1/2}
\\ \nonumber
& & - C_2 \left( \int \left| \Delta ^{\lfloor \alpha/2 \rfloor}
\nabla^{2\left(\alpha/2-\lfloor \alpha/2 \rfloor \right)} u \right|^2 dx\right)^{(k+1)/2}.
\end{eqnarray}
Again, the validity of the H\"{o}lder inequalities in the first step and of the Sobolev embeddings in the
second can be checked in the proof of lemma~\ref{definition}.

{\it Step 3. Case $k = N/2$.}

Finally, in the critical case $k = N/2$ we find
\begin{eqnarray} \nonumber
J[u] &\ge& \frac{1}{2} \int \left| \Delta ^{\lfloor \alpha/2 \rfloor}
\nabla^{2\left(\alpha/2-\lfloor \alpha/2 \rfloor \right)} u \right|^2 dx
\\ \nonumber
& & - \left|\lambda\right| \|f\|_{h^1_r(\Omega)} \|u\|_{\mathrm{bmo}_z(\bar{\Omega})} - \frac{1}{k+1} \|u\|_{\mathrm{bmo}_r(\Omega)} \|S_k[u]\|_{h^1_z(\bar{\Omega})}
\\ \nonumber
&\ge& \frac{1}{2} \int \left| \Delta ^{\lfloor \alpha/2 \rfloor}
\nabla^{2\left(\alpha/2-\lfloor \alpha/2 \rfloor \right)} u \right|^2 dx
\\ \nonumber
& & - C_1 \left|\lambda\right| \|f\|_{h^1_r(\Omega)} \left( \int \left| \Delta ^{\lfloor \alpha/2 \rfloor}
\nabla^{2\left(\alpha/2-\lfloor \alpha/2 \rfloor \right)} u \right|^2 dx \right)^{1/2}
\\ \nonumber
& & - C_2 \left( \int \left| \Delta ^{\lfloor \alpha/2 \rfloor}
\nabla^{2\left(\alpha/2-\lfloor \alpha/2 \rfloor \right)} u \right|^2 dx\right)^{(k+1)/2},
\end{eqnarray}
where we have used, in the first step, the inequalities obtained by duality in the proof of lemma~\ref{localhardy} and, in the second step,
the inclusions $W_0^{\alpha,2}(\Omega) \subset \mathrm{bmo}_z(\bar{\Omega})$ and
$W_0^{\alpha,2}(\Omega) \subset \mathrm{bmo}_r(\Omega)$ from the proof of lemma~\ref{localhardy} and
the implication $u \in W_0^{\alpha,2}(\Omega) \Rightarrow S_k[u] \in h^1_z(\bar{\Omega})$, from the Sobolev embedding at the
beginning of the proof of lemma~\ref{localhardy} and the statement of lemma~\ref{hardyz}.
\end{proof}

\begin{lemma}\label{lembound}
There exist functions $\varphi,\psi \in W^{\alpha,2}_0(\Omega)$ such that
\begin{itemize}
\item $\lambda \int f \, \varphi \, dx > 0$,
\item $(-1)^k \int \psi \, S_k[\psi] \, dx > 0$.
\end{itemize}
\end{lemma}

\begin{proof}
For the function $\varphi$ we can choose a suitable mollification of $f$ times $\lambda$.
In order to find a suitable $\psi$ lets consider a ball $\mathcal{B} \Subset \Omega$ and an auxiliary function $\bar{\psi}$.
For $\bar{\psi}$ we choose a smooth function compactly supported
in $\mathcal{B}$ such that it is radially symmetric with respect to the center of $\mathcal{B}$, its global maximum lies at the center of $\mathcal{B}$
and it is strictly decreasing for increasing radius (from the center of $\mathcal{B}$ again). We also choose $\bar{\psi}$ such that its Hessian matrix is
negative definite when evaluated at the maximum.
Now we claim $\int S_k[\bar{\psi}] \, dx = 0$; this is a direct consequence of the divergence form of the $k-$Hessian operator.
Next we claim that $S_k[\bar{\psi}]<0$ in a neighborhood of the maximum of $\bar{\psi}$; this is a direct consequence of the dependence of $S_k[u]$ on
the eigenvalues of the Hessian matrix and of the non-degeneracy assumptions on the maximum of $\bar{\psi}$. Both claims immediately imply that
$\int \bar{\psi} \, S_k[\bar{\psi}] \, dx < 0$. So for odd $k$ we can choose $\psi=\bar{\psi}$ and for even $k$ we can choose
$\psi=-\bar{\psi}$ (i.~e. $\psi=(-1)^{k+1} \, \bar{\psi} \,\, \forall \, k$).
\end{proof}

\begin{corollary}\label{corunbound}
Let $t \in \mathbb{R}_+$, then $J[t \varphi] <0$ for $t$ small enough and $J[t \psi]<0$ for $t$ large enough.
\end{corollary}

\subsection{Palais-Smale compactness condition for $N/2 < k \le N$}

\begin{definition}
A sequence $\{u_n\}_{n \in \mathbb{N}} \subset W^{\alpha,2}_0(\Omega)$ such that
\begin{itemize}
\item $J[u_n] \to K \in \mathbb{R}, \, n \to \infty$,
\item $J'[u_n] \to 0 \quad \mathrm{in} \quad W^{-\alpha,2}(\Omega)$.
\end{itemize}
is called a Palais-Smale sequence to functional $J[u]$.
\end{definition}

We will assume the existence of a Palais-Smale sequence to functional $J[u]$ and prove a precompactness condition for it
in the case $N/2 < k \le N$.

\begin{proposition}\label{palais}
If a Palais-Smale sequence $\{u_n\}_{n \in \mathbb{N}}$ to functional $J[u]$ is bounded in $W_0^{\alpha,2}(\Omega)$
then there exists a subsequence $\{u_{n_j}\}_{n_j \in \mathbb{N}}$ that converges strongly in $W_0^{\alpha,2}(\Omega)$.
\end{proposition}

\begin{proof}
If $\{u_n\}_{n \in \mathbb{N}}$ is uniformly bounded in $W_0^{\alpha,2}(\Omega)$ then, up to passing to a suitable subsequence,
we have the following convergence properties:
\begin{itemize}
\item[(A)] $u_{n} \rightharpoonup u$ weakly in $W_0^{\alpha,2}(\Omega)$,
\item[(B)] $u_{n} \rightarrow u$ uniformly in $\Omega$.
\end{itemize}
Convergence property (A) follows from the fact that any sequence bounded in a Hilbert space is weakly precompact.
On the other hand, convergence property (B) follows from two facts, the first one is the embedding of the Sobolev space $W_0^{\alpha,2}(\Omega)$
into the space
\begin{itemize}
\item $C^{0,\gamma}(\bar{\Omega}) \,\, \forall \, \gamma < 1 $ when $N$ is even,
\item $C^{0,1/2}(\bar{\Omega})$ when $N$ is odd and $k \le \lfloor 2N/3 \rfloor$,
\item $C^{1,1/2}(\bar{\Omega})$ when $N$ is odd and $k \ge \lfloor 2N/3+1 \rfloor$,
\end{itemize}
that implies boundedness and uniform equicontinuity of the sequence $\{u_n\}_{n \in \mathbb{N}}$, and the second fact is the Arzel\`a-Ascoli theorem.

Now consider the weak form of our equation
$$
(-1)^\alpha \Delta^\alpha u_n = (-1)^k S_k[u_n] + \lambda f + y_n, \qquad y_n \xrightarrow[n \to \infty]{} 0 \quad \mathrm{in} \quad W^{-\alpha,2}(\Omega).
$$
The pairing of $J'[u_n]$ with $(u_n-u)$ yields
\begin{eqnarray} \nonumber
& & \int \Delta ^{\lfloor \alpha/2 \rfloor}
\nabla^{2\left(\alpha/2-\lfloor \alpha/2 \rfloor \right)} u_n \,\, \Delta ^{\lfloor \alpha/2 \rfloor}
\nabla^{2\left(\alpha/2-\lfloor \alpha/2 \rfloor \right)} (u_n -u) dx
\\ \nonumber
&=& (-1)^k \int S_k[u_n](u_n-u) dx + \lambda \int f (u_n-u) dx + \langle y_n, u_n-u \rangle.
\end{eqnarray}
The second and third terms on the right hand side vanish in the limit $n \to \infty$ due to convergence property (A)
and the first term vanishes due to convergence property (B).
Furthermore we have
\begin{equation*}
\int \Delta ^{\lfloor \alpha/2 \rfloor}
\nabla^{2\left(\alpha/2-\lfloor \alpha/2 \rfloor \right)} u \,\, \Delta ^{\lfloor \alpha/2 \rfloor}
\nabla^{2\left(\alpha/2-\lfloor \alpha/2 \rfloor \right)} (u_n -u) dx  \xrightarrow[n \to \infty]{} 0,
\end{equation*}
due to convergence property (A); adding the last two results we find
\begin{equation}\nonumber
\int \left| \Delta ^{\lfloor \alpha/2 \rfloor}
\nabla^{2\left(\alpha/2-\lfloor \alpha/2 \rfloor \right)} u \right|^2 dx \xrightarrow[n \to \infty]{} 0.
\end{equation}
This proves the Palais-Smale condition to level $K$.
\end{proof}

\subsection{Existence and multiplicity results for $N/2 < k \le N$}
\label{existmulsec}

In this section we use in part the ideas in~\cite{GAP} to solve problems with concave–-convex nonlinearities.
In particular, note that functional~\eqref{func0} is unbounded below (this is a simple consequence of corollary~\ref{corunbound})
and then we cannot invoke standard minimization arguments
but instead we have to rely on the general theory of critical points of functionals.

\begin{theorem}\label{existmul}
There exist a $\lambda_0 > 0$ such that for $0 < |\lambda| < \lambda_0$ and $N/2 < k \le N$ problem~\eqref{dirichlet}
has at least two solutions.
\end{theorem}

\begin{proof}
According to lemma~\ref{definition} functional $J[u]$ is well defined in $W^{\alpha,2}_0(\Omega)$. It is also continuous and Gateaux differentiable, and its derivative is weak$-\ast$ continuous. This is in fact the regularity required in Ekeland's weak version
of the {\it mountain pass} theorem~\cite{ekeland}. The scheme of our proof is as follows: we will prove that one of the solutions corresponds to a
local minimum of $J[u]$ and the other one to a {\it mountain pass} critical point.

{\it Step 1. $J[u]$ has a local minimum at a negative level.}

Let $\lambda_0 > 0$ be such that, for $0 < |\lambda| < \lambda_0$,
$$
g \left( \| \Delta ^{\lfloor \alpha/2 \rfloor}\nabla^{2\left(\alpha/2-\lfloor \alpha/2 \rfloor \right)} u \|_2
\right) := G[u]
$$
attaints its positive maximum at $R_M > 0$. Let $R_0$ be the
lower positive zero of $g(R)$ and $R_0 < R_1 < R_M < R_2$ such that $g(R_1) > 0$ and $g(R_2) > 0$.
Consider also a smooth nonincreasing cutoff function $\theta:\mathbb{R}_+ \to [0,1]$ that verifies
$\theta(R)=1$ for $R \le R_0$ and $\theta(R)=0$ for $R \ge R_1$.
For $\Theta[u]=\theta \left( \| \Delta ^{\lfloor \alpha/2 \rfloor}\nabla^{2\left(\alpha/2-\lfloor \alpha/2 \rfloor \right)}
u \|_2 \right)$ we define the functional
\begin{eqnarray}\nonumber
H[u] &=& \int \left\{ -\lambda f u + \frac{1}{2} \left| \Delta ^{\lfloor \alpha/2 \rfloor}
\nabla^{2\left(\alpha/2-\lfloor \alpha/2 \rfloor \right)} u \right|^2 \right\} dx \\ \nonumber
& & - \Theta[u] \int \left\{\frac{(-1)^k}{k+1} u S_k[u] \right\} dx.
\end{eqnarray}

\begin{lemma}\label{funch}
Functional $H[u]$ fulfills the following properties:
\begin{itemize}
\item[(I)] $H[u]$ is as regular as $J[u]$.
\item[(II)] $H[u]<0 \Rightarrow \| \Delta ^{\lfloor \alpha/2 \rfloor}\nabla^{2\left(\alpha/2-\lfloor \alpha/2 \rfloor \right)}
u \|_2 < R_0$.
\item[(III)] $\| \Delta ^{\lfloor \alpha/2 \rfloor}\nabla^{2\left(\alpha/2-\lfloor \alpha/2 \rfloor \right)}
u \|_2 \le R_0 \Rightarrow H[u]=J[u]$.
\item[(IV)] We define $m:= \inf_{w \in W_0^{\alpha,2}(\Omega)} H[w]$; then $H[u]$ verifies a local Palais-Smale condition to level $m$.
\end{itemize}
\end{lemma}

\begin{proof}
The first three properties are immediate. For the fourth property see that functional $H[u]$ is bounded below by
$h \left( \| \Delta ^{\lfloor \alpha/2 \rfloor}\nabla^{2\left(\alpha/2-\lfloor \alpha/2 \rfloor \right)} u \|_2
\right)$, where
$$
h(R)=\frac{1}{2} \, R^2 - C_1 \, R - C_2 \, R^{k+1},
$$
as can be shown by means of H\"older and Sobolev inequalities exactly as in the proof of proposition~\ref{downbound}.
This boundedness, together with the negative value of $m$, implies in turn that all Palais-Smale sequences of minimizers of $H[u]$ are bounded.
Now (IV) follows by invoking proposition~\ref{palais}.
\end{proof}

From this lemma it is clear that any negative critical value of $H[u]$ is a critical value of $J[u]$ too. In particular, $m$ is
a negative critical value of $J[u]$ and therefore there exists a local minimum of this functional.

{\it Step 2. $J[u]$ has a mountain pass critical point at a positive critical level.}

We have already checked in section~\ref{geometryj} that $J[u]$ fulfills the geometric constraints imposed by the
{\it mountain pass} theorem~\cite{ar,ekeland}. Denote by $u_m$ the local minimum which existence was proven in the previous step.
For $v \in W^{\alpha,2}_0(\Omega)$ such that $\| \Delta ^{\lfloor \alpha/2 \rfloor}\nabla^{2\left(\alpha/2-\lfloor \alpha/2 \rfloor \right)}
u \|_2 > R_M$ and $J[v] < J[u_m]$ we define
$$\Gamma=\left\{ \left. \gamma(t) \in C \left([0,1], W^{\alpha,2}_0(\Omega) \right) \, \right| \, \gamma(0)=u_m,\, \gamma(1)=v \right\},$$
as well as the minimax value
$$K=\inf_{\gamma \in \Gamma} \max_{t\in[0,1]} J[\gamma(t)].$$
The existence of a Palais-Smale sequence to level $K$, that is a sequence
$\{u_n\}_{n\in \mathbb{N}}\subset W^{\alpha,2}_0(\Omega)$ that fulfills
\begin{itemize}
\item $J[u_n] \to K$ when $n \to \infty$,
\item $J'[u_n] \to 0$ in $W^{-\alpha,2}(\Omega)$,
\end{itemize}
follows from Ekeland's variational principle~\cite{ekeland2}. Our next step is proving that all Palais-Smale sequences to level
$K$ are uniformly bounded in $W_0^{\alpha,2}(\Omega)$.
A combination of integration by parts, H\"older, Sobolev and duality inequalities yields
\begin{eqnarray} \nonumber
K + o(1) &=& J[u_n] -\frac{1}{k+1} \langle J'(u_n), u_n \rangle + \frac{1}{k+1} \langle y_n, u_n \rangle
\\ \nonumber
&\ge& \left( \frac{1}{2}-\frac{1}{k+1} \right) \int \left| \Delta ^{\lfloor \alpha/2 \rfloor}
\nabla^{2\left(\alpha/2-\lfloor \alpha/2 \rfloor \right)} u \right|^2 dx
\\ \nonumber
& & -|\lambda| C' \! \left(1-\frac{1}{k+1} \right) \! \|f\|_{1} \!
\left( \int \left| \Delta ^{\lfloor \alpha/2 \rfloor}
\nabla^{2\left(\alpha/2-\lfloor \alpha/2 \rfloor \right)} u \right|^2 dx \right)^{1/2}
\\ \nonumber
& & -\frac{1}{k+1} C'' \|y_n\|_{W^{-\alpha,2}(\Omega)} \! \left( \int \left| \Delta ^{\lfloor \alpha/2 \rfloor}
\nabla^{2\left(\alpha/2-\lfloor \alpha/2 \rfloor \right)} u \right|^2 dx \right)^{1/2} \! ,
\end{eqnarray}
for suitable positive constants $C'$ and $C''$.
For small enough $|\lambda|$ this proves the sequence is bounded. Now lemma~\ref{palais} implies that $J[u]$ fulfills the Palais-Smale
condition to level $K$ and consequently:
\begin{itemize}
\item[$\boldsymbol{\ast}$] $J[u_\star] = \lim_{n \to \infty} J[u_n] = K$,
\item[$\boldsymbol{\ast}$] $J'[u_\star] = 0$.
\end{itemize}
The last equality is equivalent to
$$
(-1)^\alpha \Delta^\alpha u_\star = (-1)^k S_k[u_\star] + \lambda f,
$$
for $u_\star \in W^{\alpha,2}_0(\Omega)$.
Finally note that $u_\star$ is necessarily different from $u_m$, since $J[u_m]<0$ and $J[u_\star]>0$.
Therefore we conclude that $u_\star$ is a {\it mountain pass} sort of solution to problem~\eqref{dirichlet}.
\end{proof}

\subsection{Existence and local uniqueness results for $2 \le k \le N/2$}
\label{existunisec}

In this case there is not enough compactness to prove the Palais-Smale condition analogously to what was done in proposition~\ref{palais}, i.e.,
invoking the Arzel\`a-Ascoli theorem.
However, we do not need such a strong condition to prove the existence of local minima, we just need our functional to be
weakly lower semicontinuous. We will use this fact to prove the existence of an isolated solution to problem~\eqref{dirichlet}.

\begin{remark}
Although there is not enough compactness to prove the Palais-Smale condition
invoking the Arzel\`a-Ascoli theorem in the range $2 \le k \le N/2$, it is still possible
to prove the existence of the \emph{mountain pass} sort of solution
employing different techniques, see section~\ref{weakcontsec}.
\end{remark}

\begin{definition}
Let $\mathfrak{u}$ be a solution to problem~\eqref{dirichlet}.
If there exists a $\varrho > 0$ such that this solution is unique in the ball
$$
\beth_\varrho(\mathfrak{u})=\left\{ \varphi \in W^{\alpha,2}_0(\Omega) : \| \Delta ^{\lfloor \alpha/2 \rfloor}
\nabla^{2\left(\alpha/2-\lfloor \alpha/2 \rfloor \right)} (\mathfrak{u} - \varphi) \|_2 \le \varrho \right\},
$$
then we say that $\mathfrak{u}$ is an isolated solution.
\end{definition}

\begin{theorem}\label{existuni}
There exists a $\lambda_1 >0$ such that for $0 < |\lambda| < \lambda_1$ and $2 \le k \le N/2$ problem~\eqref{dirichlet}
has at least one solution $u_m$. Furthermore, this solution is isolated.
\end{theorem}

\begin{proof}
Consider the functional $H[u]$ which properties were proved in lemma~\ref{funch}.
We will focus on the region $\| \Delta ^{\lfloor \alpha/2 \rfloor}
\nabla^{2\left(\alpha/2-\lfloor \alpha/2 \rfloor \right)} u \|_2 \le R_0$.
In this case a calculation akin to that in proposition~\ref{eulag} shows that its second variation reads~\cite{wang}
\begin{eqnarray}\nonumber
\left. \frac{d^2}{dt^2} H[u+tv] \right|_{t=0} \!\! &=& \!\! (k+1) \sum_{i,j} \int v_i v_j S^{ij}_k[u] + \| \Delta ^{\lfloor \alpha/2 \rfloor}
\nabla^{2\left(\alpha/2-\lfloor \alpha/2 \rfloor \right)} v \|_2^2 \\ \nonumber
&\le& \!\! \| \Delta ^{\lfloor \alpha/2 \rfloor}
\nabla^{2\left(\alpha/2-\lfloor \alpha/2 \rfloor \right)} v \|_2^2
- C \sum_{i,j} \|v_i\|_{\tilde{q}} \|v_j\|_{\tilde{q}} \| S^{ij}_k[u] \|_{\tilde{p}} \\ \nonumber
&\le& \!\! \left( 1 - C \| \Delta ^{\lfloor \alpha/2 \rfloor}
\nabla^{2\left(\alpha/2-\lfloor \alpha/2 \rfloor \right)} u \|_2^{k-1} \right) \\ \nonumber
& & \!\! \times \| \Delta ^{\lfloor \alpha/2 \rfloor}
\nabla^{2\left(\alpha/2-\lfloor \alpha/2 \rfloor \right)} v \|_2^2 ,
\end{eqnarray}
where we have used the following H\"{o}lder inequality
\begin{eqnarray}\nonumber
\int \left| v_i v_j S^{ij}_k[u] \right| &\le& \|v_i\|_{\tilde{q}} \|v_j\|_{\tilde{q}} \| S^{ij}_k[u] \|_{\tilde{p}}, \\ \nonumber
\tilde{p} &=& \frac{N(k+1)}{(N+2)(k-1)}, \\ \nonumber
\tilde{q} &=& \frac{N(k+1)}{N-k+1},
\end{eqnarray}
in the first inequality and the Sobolev embeddings
\begin{eqnarray}\nonumber
W^{\alpha,2}_0(\Omega) &\hookrightarrow& W^{2,(k-1)\tilde{p}}_0(\Omega), \\ \nonumber
W^{2,(k-1)\tilde{p}}_0(\Omega) &\hookrightarrow& W^{1,\tilde{q}}_0(\Omega),
\end{eqnarray}
in the second. For small enough $|\lambda|$, what in turn
implies a sufficiently small $R_0$, the second variation is strictly positive, and consequently the functional is strictly convex in the region
under consideration. The statement follows as a direct consequence of this fact and the continuity of the functional.
\end{proof}

\begin{corollary}\label{unimul}
The solution corresponding to the local minimum in theorem~\ref{existmul} is isolated whenever $|\lambda|$ is small enough.
\end{corollary}

\begin{remark}\label{lambda0}
The results proven so far assume $\lambda \neq 0$. The case $\lambda=0$ can be analyzed by means of an immediate reformulation of the same
arguments. In this case we still have at least two solutions that correspond to a local minimum and to a \emph{mountain pass} critical point of our
functional. The solution corresponding to the local minimum is again isolated, and the only difference with respect to the $\lambda \neq 0$ case is that
it becomes trivial when we set this parameter to zero. This same remark applies as well to all the results to come in this paper.
\end{remark}

\section{Weak continuity and weaker notions of solution}
\label{weakcontsec}

In this section we will explore how the precise structure of our nonlinearity will allow us to improve our existence results.
Let us note that related nonlinearities have been explored in great detail in the past. The weak continuity and weak definitions of
both the Hessian and the Jacobian determinants have been studied, for instance,
in~\cite{ball,brezis1,brezis2,brezis3,brezis4,brezis5,coifman,dacorogna,giaquinta,hajlasz,iwaniec1,iwaniec2,iwaniec3,morrey,
muller1,muller2,muller3,muller4,muller5,reshetnyak},
where this list is meant by no means to be exhaustive. Some of these previous results will help us in our current task.

\subsection{Existence and multiplicity results for $2 \le k < N/2$}

We start proving a technical result that will in turn allow us to prove the weak continuity of our nonlinearity.
Precisely this will be the substitute of uniform equicontinuity and the Arzel\`a-Ascoli theorem in our proof of existence
of the \emph{mountain pass} type of solution. We denote generically by $\mathfrak{M}_k(D^2 u)$ an arbitrary minor of order $k$ of the Hessian matrix.

\begin{lemma}\label{leminor}
Every minor $\mathfrak{M}_k(D^2 u)$ of order $k$, $1 \le k \le N$, of the Hessian matrix $D^2 u$ is weakly$-*$ continuous in the sense of measures, i. e., if
$$
u_n \rightharpoonup u \qquad \text{weakly in} \quad W^{2,k}(\Omega),
$$
then
$$
\mathfrak{M}_k(D^2 u_n) \overset{*}{\rightharpoonup} \mathfrak{M}_k(D^2 u) \qquad \text{weakly$-*$ in} \quad \mathcal{M}(\Omega),
$$
where $\mathcal{M}(\Omega)$ is the space of signed Radon measures on $\Omega$ with finite mass.
Furthermore, if
$$
u_n \rightharpoonup u \qquad \text{weakly in} \quad W^{2,k \wp}(\Omega),
$$
for some $\wp >1$, then
$$
\mathfrak{M}_k(D^2 u_n) \rightharpoonup \mathfrak{M}_k(D^2 u) \qquad \text{weakly in} \quad L^{\wp}(\Omega).
$$
\end{lemma}

\begin{proof}

{\it Step 1.}

Our first step is proving weak continuity in the sense of distributions, i. e., if
$$
u_n \rightharpoonup u \qquad \text{weakly in} \quad W^{2,k}(\Omega),
$$
then
$$
\mathfrak{M}_k(D^2 u_n) \rightarrow \mathfrak{M}_k(D^2 u) \qquad \text{in} \quad \mathcal{D}'(\Omega).
$$
The proof follows by induction. Linearity guarantees that
all minors of order $1$ (that is, all entries of the matrix $D^2 u$) are weakly continuous.
From now on we assume $u \in C^\infty(\Omega)$ and argue by approximation when necessary.
Note also that $W^{2,k_2}(\Omega) \subset W^{2,k_1}(\Omega)$ whenever $k_1 < k_2$.
Our induction will be based in proving that the
weak continuity of all minors of order $k-1$ implies the weak continuity of all minors of order $k$. Select one such minor:
$$
\mathfrak{M}_k(D^2 u) = \det \! \left[ \left( \frac{\partial^2 u}{\partial x_p \partial x_q} \right)_{p \in P, q \in Q} \right],
$$
where $P,Q$ are subsets of $\{1, \cdots, N\}$ with cardinality $k$. Let $\varphi \in C_0^\infty(\Omega)$, integrating by parts we find
\begin{eqnarray}\label{intpartdet}
& & \int \det \! \left[ \left( \frac{\partial^2 u}{\partial x_p \partial x_q} \right)_{p \in P, q \in Q} \right] \varphi \, dx \\ \nonumber
&=& - \sum_{q \in Q} \int \frac{\partial u}{\partial x_p} \, \mathrm{cof} \!
\left[ \left( \frac{\partial^2 u}{\partial x_{\bar{p}} \partial x_{\bar{q}}} \right)_{\bar{p} \in P, \bar{q} \in Q} \right]_{pq}
\frac{\partial \varphi}{\partial x_q} \, dx,
\quad p \in P,
\end{eqnarray}
where we have used
$$
\det \! \left[ \left( \frac{\partial^2 u}{\partial x_p \partial x_q} \right)_{p \in P, q \in Q} \right]=
\sum_{q \in Q} \frac{\partial}{\partial x_q}
\left\{ \frac{\partial u}{\partial x_p} \, \mathrm{cof}\! \left[ \left( \frac{\partial^2 u}{\partial x_{\bar{p}} \partial x_{\bar{q}}}
\right)_{\bar{p} \in P, \bar{q} \in Q} \right]_{pq} \right\},
$$
$p \in P$. Now for $u_n \rightharpoonup u$ weakly in $W^{2,k}(\Omega)$ we have
\begin{eqnarray}\nonumber
& & \lim_{n \to \infty} \int \det \! \left[ \left( \frac{\partial^2 u_n}{\partial x_p \partial x_q} \right)_{p \in P, q \in Q} \right] \varphi \, dx \\
\nonumber
&=& \lim_{n \to \infty} \,\, - \sum_{q \in Q} \int \frac{\partial u_n}{\partial x_p} \, \mathrm{cof} \!
\left[ \left( \frac{\partial^2 u_n}{\partial x_{\bar{p}} \partial x_{\bar{q}}} \right)_{\bar{p} \in P, \bar{q} \in Q} \right]_{pq}
\frac{\partial \varphi}{\partial x_q} \, dx \\ \nonumber
&=& - \sum_{q \in Q} \int \frac{\partial u}{\partial x_p} \, \mathrm{cof} \!
\left[ \left( \frac{\partial^2 u}{\partial x_{\bar{p}} \partial x_{\bar{q}}} \right)_{\bar{p} \in P, \bar{q} \in Q} \right]_{pq}
\frac{\partial \varphi}{\partial x_q} \, dx \\ \nonumber
&=& \int \det \! \left[ \left( \frac{\partial^2 u}{\partial x_p \partial x_q} \right)_{p \in P, q \in Q} \right] \varphi \, dx,
\end{eqnarray}
$p \in P$, where we have used~\eqref{intpartdet} in the first and third steps, and $\partial_{x_p} u_n \in L^{k}(\Omega)$,
$$
\mathrm{cof} \!
\left[ \left( \frac{\partial^2 u_n}{\partial x_{\bar{p}} \partial x_{\bar{q}}} \right)_{\bar{p} \in P, \bar{q} \in Q} \right]_{pq}
\in L^{k/(k-1)}(\Omega),
$$
the Sobolev embedding\footnote{This is actually an abuse of notation. In the case $N=k$, of course,
one should interpret this as the embedding into the Sobolev space $W^{1,r} \, \forall \, r <\infty$.} $W^{2,k} \hookrightarrow W^{1,Nk/(N-k)}$,
the inequality $Nk/(N-k)> k$, the Rellich-Kondrachov theorem that guarantees $u_n \to u$ strongly
in $W^{1,k}$ provided $u_n \rightharpoonup u$ weakly in $W^{2,k}$, and
$$
\mathrm{cof} \!
\left[ \left( \frac{\partial^2 u_n}{\partial x_{\bar{p}} \partial x_{\bar{q}}} \right)_{\bar{p} \in P, \bar{q} \in Q} \right]_{pq}
\! \rightharpoonup \,
\mathrm{cof} \!
\left[ \left( \frac{\partial^2 u}{\partial x_{\bar{p}} \partial x_{\bar{q}}} \right)_{\bar{p} \in P, \bar{q} \in Q} \right]_{pq},
$$
weakly in $L^{k/(k-1)}(\Omega)$,
which is actually the induction hypothesis, together with the product of weakly and strongly converging sequences in the second.
The proof in case $u_n \rightharpoonup u$ weakly in $W^{2,k \wp}(\Omega)$ follows analogously.

{\it Step 2.}

Once convergence in the sense of distributions is proven we just need to realize that, when $u_n \in W^{2,k}(\Omega)$, then $\mathfrak{M}_k(D^2 u_n)$
is bounded in $L^1(\Omega)$ and consequently this sequence converges weakly$-*$ in $\mathcal{M}(\Omega)$. By uniqueness of weak limit we conclude.
Analogously, if $u_n \in W^{2,k \wp}(\Omega)$ then $\mathfrak{M}_k(D^2 u_n)$ is bounded in $L^\wp(\Omega)$. Thus the sequence converges
weakly in $L^\wp$ and by uniqueness of weak limit we conclude.
\end{proof}

\begin{proposition}\label{weakcont}
Let $2 \le k < N/2$.
Then $S_k[u]$ is weakly continuous in $L^{p^*}(\Omega)$, that is, if
$$
u_n \rightharpoonup u \qquad \text{weakly in} \quad W^{2,k p^*}(\Omega),
$$
then
$$
S_k[u_n] \rightharpoonup S[u] \qquad \text{weakly in} \quad L^{p^*}(\Omega).
$$
\end{proposition}

\begin{proof}
Consider $u \in W^{2,k p^*}(\Omega)$ and $v \in C^\infty_0(\Omega)$.
We compute
\begin{equation}\label{eqweakcont}
\int S_k[u]v \, dx = \frac{1}{k} \sum_{i,j} \int v_i u_j S^{ij}_k[u] \, dx,
\end{equation}
where we have used integration by parts and the divergence form of the $k-$Hessian.
Sobolev embeddings and H\"{o}lder inequalities reveal that both hands of~\eqref{eqweakcont} are well defined.
Note also that $\{u_n\}_n$ is bounded in $W^{2,k p^*}(\Omega)$ and consequently
$$
u_n \rightarrow u \qquad \text{strongly in} \quad W^{1,\tilde{q}/2}(\Omega),
$$
by the Rellich-Kondrachov theorem.
Therefore, by product of weakly and strongly converging sequences, we get
$$
\lim_{n \to \infty} \int v_i \left(u_n\right)_j S^{ij}_k[u_n] \, dx = \int v_i u_j S^{ij}_k[u] \, dx,
$$
since $u_j \in L^{\tilde{q}/2}(\Omega)$, $S^{ij}_k[u_n] \in L^{\tilde{p}}(\Omega)$ and $\frac{2}{\tilde{q}}+\frac{1}{\tilde{p}}=1$,
and $S^{ij}_k[u_n] \rightharpoonup S^{ij}_k[u]$ weakly in $L^{\tilde{p}}(\Omega)$, which is a consequence
of lemma~\ref{leminor} and the fact that $S^{ij}_k[u_n]$ is a linear combination of minors of order $k-1$.
Now by~\eqref{eqweakcont} we find
$$
\lim_{n \to \infty} \int S_k[u_n]v \, dx = \int S_k[u]v \, dx.
$$
Since $S_k[u_n]$ is bounded in $L^{p^*}(\Omega)$ it admits a weakly converging subsequence in this space,
and by uniqueness of weak limit we conclude.
\end{proof}

The following result is the counterpart of theorem~\ref{existmul} in the range $2 \le k < N/2$; in fact, the existence of the solution corresponding to
the local minimum of our functional was already proven in theorem~\ref{existuni}. The main difficulty in proving the existence of the \emph{mountain
pass} critical point was the lack of compactness to prove the Palais-Smale condition invoking the Arzel\`a-Ascoli theorem,
see section~\ref{existunisec}. In the present case, the existence of the \emph{mountain pass} sort of solution follows from an argument
akin to that in section~\ref{existmulsec} but replacing the Arzel\`a-Ascoli theorem by the weak continuity proven in proposition~\ref{weakcont}.

\begin{theorem}\label{existwc}
There exist a $\lambda_0 > 0$ such that for $0 < |\lambda| < \lambda_0$ and $2 \le k < N/2$ problem~\eqref{dirichlet}
has at least two solutions.
\end{theorem}

\begin{proof}
{\it Step 1.}

As in the case of theorem~\ref{existmul}, the existence of a Palais-Smale sequence to level $K$, that is a sequence
$\{u_n\}_{n\in \mathbb{N}}\subset W^{\alpha,2}_0(\Omega)$ that fulfills
\begin{itemize}
\item $J[u_n] \to K$ when $n \to \infty$,
\item $J'[u_n] \to 0$ in $W^{-\alpha,2}(\Omega)$,
\end{itemize}
follows from Ekeland's variational principle~\cite{ekeland2}. As in proposition~\ref{palais}, we consider the weak form of our equation
$$
(-1)^\alpha \Delta^\alpha u_n = (-1)^k S_k[u_n] + \lambda f + y_n, \qquad y_n \xrightarrow[n \to \infty]{} 0 \quad \mathrm{in} \quad W^{-\alpha,2}(\Omega).
$$

{\it Step 2.}

The next step is proving that all Palais-Smale sequences to level
$K$ are uniformly bounded in $W_0^{\alpha,2}(\Omega)$.
A combination of integration by parts, H\"older, Sobolev and duality inequalities yields
\begin{eqnarray} \nonumber
K + o(1) &=& J[u_n] -\frac{1}{k+1} \langle J'(u_n), u_n \rangle + \frac{1}{k+1} \langle y_n, u_n \rangle
\\ \nonumber
&\ge& \left( \frac{1}{2}-\frac{1}{k+1} \right) \int \left| \Delta ^{\lfloor \alpha/2 \rfloor}
\nabla^{2\left(\alpha/2-\lfloor \alpha/2 \rfloor \right)} u \right|^2 dx -|\lambda| C'
\\ \nonumber
& & \times \left(1-\frac{1}{k+1} \right) \|f\|_{p^*}
\left( \int \left| \Delta ^{\lfloor \alpha/2 \rfloor}
\nabla^{2\left(\alpha/2-\lfloor \alpha/2 \rfloor \right)} u \right|^2 dx \right)^{1/2}
\\ \nonumber
& & -\frac{1}{k+1} C'' \|y_n\|_{W^{-\alpha,2}(\Omega)} \! \left( \int \left| \Delta ^{\lfloor \alpha/2 \rfloor}
\nabla^{2\left(\alpha/2-\lfloor \alpha/2 \rfloor \right)} u \right|^2 dx \right)^{1/2} \! ,
\end{eqnarray}
for suitable positive constants $C'$ and $C''$.
For small enough $|\lambda|$ this proves the sequence is bounded.

{\it Step 3.}

Since we proved in Step $2$ that the Palais-Smale sequences are uniformly bounded
we can invoke weak continuity to find
$$
\lim_{n \to \infty} \left\langle -\lambda f + (-1)^\alpha \Delta^\alpha u_n - (-1)^k S_k[u_n], w \right\rangle =0,
$$
for all $w \in W_0^{\alpha,2}(\Omega)$ implies
\begin{itemize}
\item[$\boldsymbol{\ast}$] $J'[u_\star] = \lim_{n \to \infty} J'[u_n] = 0$,
\end{itemize}
or equivalently
$$
(-1)^\alpha \Delta^\alpha u_\star = (-1)^k S_k[u_\star] + \lambda f,
$$
for $u_\star \in W^{\alpha,2}_0(\Omega)$.
Note that $u_\star$ is necessarily different from $u_m$ since the mountain pass geometry is independent of the presence of the local minimum,
as a construction akin to that in step $1$ of theorem~\ref{existmul} guarantees.
Therefore the existence of a second variational solution follows from theorem~\ref{weakcont},
and the embedding $W_0^{\alpha,2}(\Omega) \subset L^{q^*}(\Omega)$.
\end{proof}

\subsection{Existence and multiplicity results for $k = N/2$}

The existence and multiplicity results have now been proven for all possible cases except for the borderline one $k=N/2$. The particularity of this
problem is that it requires the use of harmonic analytical tools as in its characterization as an Euler-Lagrange equation of a suitable functional
in section~\ref{variationalset}. Our approach requires expanding the functional setting presented so far.

\begin{definition}
We define $\mathrm{vmo}_z(\bar{\Omega})$ as the closure of $C_0(\Omega)$ in $\mathrm{bmo}_z(\bar{\Omega})$,
with $\| f \|_{\mathrm{vmo}_z(\bar{\Omega})}=\| f \|_{\mathrm{bmo}_z(\bar{\Omega})} \, \forall \, f \in \mathrm{vmo}_z(\bar{\Omega})$.
\end{definition}

We begin stating a result that concerns the weak continuity of the nonlinearity under study.
In particular, note that it is not enough invoking lemma~\ref{leminor} above.

\begin{proposition}\label{weakcontast}
Let $k = N/2$.
Then $S_k[u]$ is weakly$-*$ continuous in $h^1_r(\Omega)$, that is, if
$$
u_n \rightharpoonup u \qquad \text{weakly in} \quad W_0^{2,N/2}(\Omega),
$$
then
$$
S_k[u_n] \overset{*}{\rightharpoonup} S[u] \qquad \text{weakly$-*$ in} \quad h^1_r(\Omega).
$$
\end{proposition}

\begin{proof}
First of all note that $S_k[u_n], S_k[u] \in h^1_r(\Omega)$ for $u_n, u \in W_0^{2,N/2}(\Omega)$ as a consequence of lemma~\ref{hardyz}
and the inclusion $h^1_z(\bar{\Omega}) \subset h^1_r(\Omega)$~\cite{der}.
The statement of the proposition says that
$$
\lim_{n \to \infty} \int w \, S_k[u_n] \, dx = \int w \, S_k[u] \, dx,
$$
for all $w \in \mathrm{vmo}_z(\bar{\Omega})$, since $[\mathrm{vmo}_z(\bar{\Omega})]^*=h^1_r(\Omega)$.
Given that $C_0(\Omega)$ is dense in $\mathrm{vmo}_z(\bar{\Omega})$ we may choose
an approximating family $w_\epsilon$ of $w$ such that
$\| w-w_\epsilon \|_{\mathrm{vmo}_z(\bar{\Omega})} \le \epsilon$ for any $\epsilon >0$.
Now we compute
\begin{eqnarray}\nonumber
\int w \, S_k[u_n] \, dx - \int w \, S_k[u] \, dx &=& \int w_\epsilon \, S_k[u_n] \, dx - \int w_\epsilon \, S_k[u] \, dx \\ \nonumber
& & + \int (w-w_\epsilon) \, S_k[u_n] \, dx \\ \nonumber & & - \int (w-w_\epsilon) \, S_k[u] \, dx.
\end{eqnarray}
Taking into account that, if
$$
u_n \rightharpoonup u \qquad \text{weakly in} \quad W^{2,N/2}(\Omega),
$$
then
$$
S_k[u_n] \overset{*}{\rightharpoonup} S[u] \qquad \text{weakly$-*$ in} \quad \mathcal{M}(\Omega),
$$
which is a consequence of lemma~\ref{leminor} and the fact that $S_k[u_n]$ is a linear combination of minors of order $k$ of the Hessian matrix of $u$,
we find that
\begin{eqnarray}\nonumber
\left| \int w \, S_k[u_n] \, dx - \int w \, S_k[u] \, dx \right| &\le&
\left( \| S_k[u_n] \|_{h^1_r(\Omega)} + \| S_k[u] \|_{h^1_r(\Omega)} \right) \\ \nonumber
& & \times \| w-w_\epsilon \|_{\mathrm{vmo}_z(\bar{\Omega})} \\ \nonumber
& & + \left| \int w_\epsilon \, S_k[u_n] \, dx - \int w_\epsilon \, S_k[u] \, dx \right|,
\end{eqnarray}
and
\begin{equation}\nonumber
\limsup_{n \to \infty} \left| \int w \, S_k[u_n] \, dx - \int w \, S_k[u] \, dx \right| \le C \epsilon.
\end{equation}
The arbitrariness of $\epsilon$ concludes the proof.
\end{proof}

\begin{theorem}\label{existwc2}
There exist a $\lambda_0 > 0$ such that for $0 < |\lambda| < \lambda_0$ and $k = N/2$ problem~\eqref{dirichlet}
has at least two solutions.
\end{theorem}

\begin{proof}
{\it Step 1.}

As in previous cases, the existence of a Palais-Smale sequence to level $K$, that is a sequence
$\{u_n\}_{n\in \mathbb{N}}\subset W^{\alpha,2}_0(\Omega)$ that fulfills
\begin{itemize}
\item $J[u_n] \to K$ when $n \to \infty$,
\item $J'[u_n] \to 0$ in $W^{-\alpha,2}(\Omega)$,
\end{itemize}
follows from Ekeland's variational principle~\cite{ekeland2}. As in proposition~\ref{palais}, we consider the weak form of our equation
$$
(-1)^\alpha \Delta^\alpha u_n = (-1)^k S_k[u_n] + \lambda f + y_n, \qquad y_n \xrightarrow[n \to \infty]{} 0 \quad \mathrm{in} \quad W^{-\alpha,2}(\Omega).
$$

{\it Step 2.}

Now we show that all Palais-Smale sequences to level
$K$ are uniformly bounded in $W_0^{N/2,2}(\Omega)$.
By means of integration by parts, H\"older, Sobolev and duality inequalities we get
\begin{eqnarray} \nonumber
K + o(1) &=& J[u_n] -\frac{1}{k+1} \langle J'(u_n), u_n \rangle + \frac{1}{k+1} \langle y_n, u_n \rangle
\\ \nonumber
&\ge& \left( \frac{1}{2}-\frac{1}{k+1} \right) \int \left| \Delta ^{\lfloor \alpha/2 \rfloor}
\nabla^{2\left(\alpha/2-\lfloor \alpha/2 \rfloor \right)} u \right|^2 dx -|\lambda| C'
\\ \nonumber
& & \times \left(1-\frac{1}{k+1} \right) \|f\|_{h^1_r(\Omega)}
\left( \int \left| \Delta ^{\lfloor \alpha/2 \rfloor}
\nabla^{2\left(\alpha/2-\lfloor \alpha/2 \rfloor \right)} u \right|^2 dx \right)^{1/2}
\\ \nonumber
& & -\frac{1}{k+1} C'' \|y_n\|_{W^{-\alpha,2}(\Omega)} \! \left( \int \left| \Delta ^{\lfloor \alpha/2 \rfloor}
\nabla^{2\left(\alpha/2-\lfloor \alpha/2 \rfloor \right)} u \right|^2 dx \right)^{1/2} \! ,
\end{eqnarray}
for suitable positive constants $C'$ and $C''$.
For a sufficiently small $|\lambda|$ this implies the sequence is bounded.

{\it Step 3.}

As shown in the previous step the Palais-Smale sequences
to level $K$ are bounded, therefore the existence of a variational solution follows from the weak$-*$ continuity property of theorem~\ref{weakcontast},
the embedding $W_0^{N/2,2}(\Omega) \subset \mathrm{vmo}_z(\bar{\Omega})$ and
$$
\lim_{n \to \infty} \left\langle -\lambda f + (-1)^\alpha \Delta^\alpha u_n - (-1)^k S_k[u_n], w \right\rangle =0,
$$
for all $w \in W_0^{\alpha,2}(\Omega)$.
In other words
\begin{itemize}
\item[$\boldsymbol{\ast}$] $J'[u_\star] = \lim_{n \to \infty} J'[u_n] =0$,
\end{itemize}
or in equivalent terms
$$
(-1)^\alpha \Delta^\alpha u_\star = (-1)^k S_k[u_\star] + \lambda f,
$$
for $u_\star \in W^{N/2,2}_0(\Omega)$.
Finally note that $u_\star$ must be different from $u_m$ because the mountain pass geometry
is independent of the existence of such local minimum as a construction in the lines of that
in step $1$ of theorem~\ref{existmul} shows.
\end{proof}

\subsection{Weaker solutions: distributional divergence}

Our previous results also suggest the possibility of using weaker notions of solution. In particular we focus now on the following
boundary value problem
\begin{eqnarray}\label{weakdirichlet}
(-1)^\alpha \Delta^\alpha u = \frac{(-1)^k}{k}\sum_{i,j} \eth_{x_i}(u_{x_j} S_k^{ij}[u]) + \lambda f,
\qquad x &\in& \Omega \subset \mathbb{R}^N, \\ \nonumber
u = \partial_n u = \partial_n^2 u = \cdots = \partial_n^{\alpha-1} u = 0,
\qquad x &\in& \partial \Omega,
\end{eqnarray}
where $\eth_{x_i}$ denotes a weak derivative with respect to variable $x_i$. The existence of solutions to this problem
runs in parallel to the theory developed for problem~\eqref{dirichlet}.

\begin{remark}
Note that the value of $\alpha$ chosen coincides with the one employed in the range $2 \le k \le N/2$ for problem~\eqref{dirichlet}.
On the other hand, we also have that
$$
\alpha \le \left\lceil 2+ \frac{k-2}{2k}N \right\rceil,
$$
and it is easy to check that the inequality is strict for certain values of the parameters, so the existence result for~\eqref{weakdirichlet}
is genuinely different from the existence result for~\eqref{dirichlet} in the rank $N/2 < k \le N$. Therefore, from now on we will
concentrate on values of $k$ within this rank.
\end{remark}

\begin{proposition}
The functional
$$
J[u]= \! \! \int \! \left[ \frac{1}{2} \left| \Delta ^{\lfloor \alpha/2 \rfloor}
\nabla^{2\left(\alpha/2-\lfloor \alpha/2 \rfloor \right)} u \right|^2 \! -\lambda f u + \frac{(-1)^k}{(k+1)k} \sum_{i,j} u_i u_j S_k^{ij}[u]\right] \! dx,
$$
is well defined in $W^{\alpha,2}_0(\Omega)$ for $f \in L^1(\Omega)$ and
$$\alpha= \left\lceil \frac{Nk-N+4k}{2k+2} \right\rceil.$$
Furthermore, its critical points are solutions to boundary value problem~\eqref{weakdirichlet}.
\end{proposition}

\begin{proof}
{\it Step 1.}

We start proving that $J[u]$ is well defined in $W^{\alpha,2}_0(\Omega)$. This follows from the inequality
\begin{eqnarray}\nonumber
& & \int \! \left[ \frac{1}{2} \left| \Delta ^{\lfloor \alpha/2 \rfloor}
\nabla^{2\left(\alpha/2-\lfloor \alpha/2 \rfloor \right)} u \right|^2 \! -\lambda f u + \frac{(-1)^k}{(k+1)k} \sum_{i,j} u_i u_j S_k^{ij}[u]\right] \! dx \\
\nonumber &\le& |\lambda| \|f\|_1 \|u\|_\infty + \frac{1}{2} \| \Delta^{\lfloor \alpha/2 \rfloor}
\nabla^{2\left(\alpha/2-\lfloor \alpha/2 \rfloor \right)} u \|_2^2 \\ \nonumber
& & + \frac{1}{(k+1)k} \sum_{i,j} \|u_i\|_{\tilde{q}} \|u_j\|_{\tilde{q}} \| S^{ij}_k[u] \|_{\tilde{p}},
\end{eqnarray}
and suitable Sobolev embeddings.

{\it Step 2.}

Let $v,w \in C^\infty_0(\Omega)$, $t \in \mathbb{R}_+$ and consider $J[v + t w]$
which is well defined as a direct corollary of Step $1$.
We find
\begin{eqnarray}\nonumber
& & \left. \frac{d}{dt} J[v + t w] \right|_{t=0}
\\ \nonumber
&=& \int \bigg\{ -\lambda f w + \left[ \Delta^{\lfloor \alpha/2 \rfloor}
\nabla^{2\left(\alpha/2-\lfloor \alpha/2 \rfloor \right)} v \right]
\left[ \Delta^{\lfloor \alpha/2 \rfloor}
\nabla^{2\left(\alpha/2-\lfloor \alpha/2 \rfloor \right)} w \right]
\\ \nonumber & &
+ \frac{(-1)^k}{(k+1)k} \sum_{i,j} w_i u_j S_k^{ij}[u] + \frac{(-1)^k}{(k+1)k} \sum_{i,j} u_i w_j S_k^{ij}[u]
\\ \nonumber & &
+ \frac{(-1)^k}{(k+1)k} \sum_{i,j} u_i u_j \left. \frac{d}{dt} S_k^{ij}[v+ t w] \right|_{t=0} \bigg\} \, dx.
\end{eqnarray}
A calculation akin to that in proposition~\ref{eulag} leads to
\begin{equation*}
\left. \frac{d}{dt} J[v + t w] \right|_{t=0} \! = \!
\int \! \bigg\{ [-\lambda f + (-1)^\alpha \Delta^\alpha v] w + \frac{(-1)^k}{k}\sum_{i,j} v_{x_j} S_k^{ij}[v] \, w_{x_i} \bigg\} \, dx.
\end{equation*}
By a density argument we can take $v,w \in W^{\alpha,2}_0(\Omega)$ to conclude.
\end{proof}

\begin{theorem}\label{existwkh}
There exist a $\lambda_0 > 0$ such that for $0 < |\lambda| < \lambda_0$ and $2 \le k \le N$ problem~\eqref{weakdirichlet}
has at least two solutions.
\end{theorem}

\begin{proof}
The proof will be carried out by means of variational methods just like in the previous cases.

{\it Step 1.}

Our functional obeys the inequality
\begin{eqnarray}\nonumber
& & \int \! \left[ \frac{1}{2} \left| \Delta ^{\lfloor \alpha/2 \rfloor}
\nabla^{2\left(\alpha/2-\lfloor \alpha/2 \rfloor \right)} u \right|^2 \! -\lambda f u + \frac{(-1)^k}{(k+1)k} \sum_{i,j} u_i u_j S_k^{ij}[u]\right] \! dx \\
\nonumber &\le& -|\lambda| \|f\|_1 \|u\|_\infty + \frac{1}{2} \| \Delta ^{\lfloor \alpha/2 \rfloor}
\nabla^{2\left(\alpha/2-\lfloor \alpha/2 \rfloor \right)} u \|_2^2 \\ \nonumber
& & - \frac{1}{(k+1)k} \sum_{i,j} \|u_i\|_{\tilde{q}} \|u_j\|_{\tilde{q}} \| S^{ij}_k[u] \|_{\tilde{p}}.
\end{eqnarray}
Sobolev embeddings and the same reasoning as in proposition~\ref{downbound}, lemma~\ref{lembound} and corollary~\ref{corunbound} leads to conclude
that functional $J[u]$ fulfills the necessary geometric requirements.

{\it Step 2.}

As in the case of theorem~\ref{existmul}, the existence of a Palais-Smale sequence to level $K$, that is a sequence
$\{u_n\}_{n\in \mathbb{N}}\subset W^{\alpha,2}_0(\Omega)$ that fulfills
\begin{itemize}
\item $J[u_n] \to K$ when $n \to \infty$,
\item $J'[u_n] \to 0$ in $W^{-\alpha,2}(\Omega)$,
\end{itemize}
follows from Ekeland's variational principle~\cite{ekeland2}. As in proposition~\ref{palais}, we consider the weak form of our equation
$$
(-1)^\alpha \Delta^\alpha u_n = \frac{(-1)^k}{k}\sum_{i,j} \eth_{x_i}\{(u_n)_{x_j} S_k^{ij}[u_n]\} +
\lambda f + y_n,
$$
and $y_n \xrightarrow[n \to \infty]{} 0 \quad \mathrm{in} \quad W^{-\alpha,2}(\Omega)$.

{\it Step 3.}

Now we focus on proving the corresponding result related to weak convergence in this case.
Since $S_k^{ij}[u_n]$ is a linear combination of minors of $D^2 u$, then we have as a direct corollary of lemma~\ref{leminor} that
$$
S_k^{ij}[u_n] \rightharpoonup S_k^{ij}[u] \quad \text{weakly in} \quad L^{k/(k-1)}(\Omega).
$$
Furthermore we know $W^{2,k}(\Omega) \Subset W^{1,k}(\Omega)$, so the Rellich-Kondrachov theorem assures $(u_n)_{x_j} \to u_{x_j}$ strongly
in $L^k(\Omega)$. Now by product of weakly and strongly convergent sequences we have
$$
(u_n)_{x_j} S_k^{ij}[u_n] \rightharpoonup u_{x_j} S_k^{ij}[u] \quad \text{weakly in} \quad L^{1}(\Omega).
$$
Moreover, since $(u_n)_{x_j} S_k^{ij}[u_n]$ is bounded in $L^{\tilde{p} \tilde{q}/(\tilde{p} + \tilde{q})}(\Omega)$, then
$$
(u_n)_{x_j} S_k^{ij}[u_n] \rightharpoonup u_{x_j} S_k^{ij}[u] \quad \text{weakly in} \quad L^{\tilde{p} \tilde{q}/(\tilde{p} + \tilde{q})}(\Omega).
$$

{\it Step 4.}

The next step is proving that all Palais-Smale sequences to level
$K$ are uniformly bounded in $W_0^{\alpha,2}(\Omega)$.
A combination of integration by parts, H\"older, Sobolev and duality inequalities yields
\begin{eqnarray} \nonumber
K + o(1) &=& J[u_n] -\frac{1}{k+1} \langle J'(u_n), u_n \rangle + \frac{1}{k+1} \langle y_n, u_n \rangle
\\ \nonumber
&\ge& \left( \frac{1}{2}-\frac{1}{k+1} \right) \int \left| \Delta ^{\lfloor \alpha/2 \rfloor}
\nabla^{2\left(\alpha/2-\lfloor \alpha/2 \rfloor \right)} u \right|^2 dx -|\lambda| C'
\\ \nonumber
& & \times \left(1-\frac{1}{k+1} \right) \|f\|_{1}
\left( \int \left| \Delta ^{\lfloor \alpha/2 \rfloor}
\nabla^{2\left(\alpha/2-\lfloor \alpha/2 \rfloor \right)} u \right|^2 dx \right)^{1/2}
\\ \nonumber
& & -\frac{1}{k+1} C'' \|y_n\|_{W^{-\alpha,2}(\Omega)} \! \left( \int \left| \Delta ^{\lfloor \alpha/2 \rfloor}
\nabla^{2\left(\alpha/2-\lfloor \alpha/2 \rfloor \right)} u \right|^2 dx \right)^{1/2} \! ,
\end{eqnarray}
for suitable positive constants $C'$ and $C''$.
For small enough $|\lambda|$ this proves the sequence is bounded.

{\it Step 5.}

Since we proved in step $3$ the relevant weak continuity property for the current equation
we have
\begin{itemize}
\item[$\boldsymbol{\ast}$] $J'[u_\star] = \lim_{n \to \infty} J'[u_n] = 0$.
\end{itemize}
The last equality is equivalent to
$$
(-1)^\alpha \Delta^\alpha u_\star = \frac{(-1)^k}{k}\sum_{i,j} \eth_{x_i}\left\{(u_\star)_{x_j} S_k^{ij}[u_\star]\right\} + \lambda f,
$$
or, using the notation of theorem~\ref{mainw}, to
$$
(-1)^\alpha \Delta^\alpha u_\star = (-1)^k \mathcal{S}_k[u_\star] + \lambda f,
$$
for $u_\star \in W^{\alpha,2}_0(\Omega)$.
Note that $u_\star$ is necessarily different from $u_m$, since the same reasoning of the previous
subsection applies here as well.
In consequence $u_\star$ is a second solution to problem~\eqref{weakdirichlet}.
\end{proof}

\section{Further Results}
\label{further}

\subsection{Summable data}

In this section we consider the case in which the data $f \in L^1(\Omega)$ independently of the values of $N$ and $k$.
We will look for the optimal $\alpha$ that allows us to define our functional and build the existence theory in this case.

\begin{proposition}\label{propl1}
Let $2 \le k \le N$ and
\begin{equation}\label{newalpha}
\alpha= \left\{ \begin{array}{c}
\lceil (N+1)/2 \rceil \qquad \mathrm{if} \quad k \le \lfloor 2N/3 \rfloor \\
\lceil (N+2)/2 \rceil \qquad \mathrm{if} \quad k > \lfloor 2N/3 \rfloor
\end{array} \right. .
\end{equation}
Then functional~\eqref{func0} is well defined and its critical points correspond to solutions
to boundary value problem~\eqref{dirichlet}.
\end{proposition}

\begin{proof}
It is easy to check that the values of $\alpha$ in~\eqref{newalpha} are greater than or equal to the values of $\alpha$ in lemma~\ref{definition};
this together with the embedding $W^{\alpha,2}_0(\Omega) \hookrightarrow L^\infty(\Omega)$
guarantee that the functional is well defined. In the light of this the argument in proposition~\ref{eulag} can be exactly reproduced in
the present case.
\end{proof}

\begin{remark}
Note that the values of $\alpha$ in~\eqref{newalpha} equal those present in lemma~\ref{definition} for $k \ge N/2$.
\end{remark}

\begin{theorem}\label{unisumm}
Under the hypotheses of proposition~\ref{propl1} there exist a $\lambda_0 > 0$ such that for $0 \le |\lambda| < \lambda_0$ and $2 \le k \le N$ problem~\eqref{dirichlet}
has at least two solutions. Moreover, these solutions differ in $W^{\alpha,2}_0(\Omega)$ norm and one of them is unique in the ball
$$
\mathcal{B}=\left\{u \in W^{\alpha,2}_0(\Omega) : 0 \le \|u\|_{W^{\alpha,2}_0(\Omega)} \le \tau \right\},
$$
for some $\tau>0$.
\end{theorem}

\begin{proof}
The proof mimics exactly that of section~\ref{variational} for the case $N/2 < k \le N$.
\end{proof}

\begin{remark}
Just like in all previous cases, the solution $u_m$ corresponding to the local minimum of the functional is the one we know is isolated.
This solution is related to the mountain pass $u_\star$ one in the following way:
$$\| \Delta ^{\lfloor \alpha/2 \rfloor}\nabla^{2\left(\alpha/2-\lfloor \alpha/2 \rfloor \right)} u_m \|_2 <
\| \Delta ^{\lfloor \alpha/2 \rfloor}\nabla^{2\left(\alpha/2-\lfloor \alpha/2 \rfloor \right)} u_\star \|_2.$$
\end{remark}

\begin{remark}
The requirement $\alpha \ge \lceil (N+1)/2 \rceil$ is optimal in order to prove the Palais-Smale condition
via the Arzel\`a-Ascoli theorem and the boundedness of $u$ in $L^\infty(\Omega)$.
Therefore one could say hypothesis~\eqref{newalpha} is sharp in this respect.
\end{remark}

\subsection{Weak$-*$ continuity of $S_k[u]$ in the range $N/2 \le k \le N$}

Although not needed for the proof of the Palais-Smale condition in the range $N/2 < k \le N$,
we prove here a result that shows the weak continuity of $S_k[u]$ for all $N/2 \le k \le N$.

\begin{lemma}\label{hardyz2}
For any $u \in W^{2,k}_0(\Omega)$, $N/2 \le k \le N$, we have that $S_{k}[u] \in h_z^1(\bar{\Omega})$.
\end{lemma}

\begin{proof}
The proof follows identically the one of lemma~\ref{hardyz}.
\end{proof}

\begin{theorem}\label{weakcont2}
Let $N/2 \le k \le N$,
then $S_k[u]$ is weakly continuous in $h^1_r(\Omega)$ and in $\mathcal{M}(\Omega)$. That is, if
$$
u_n \rightharpoonup u \qquad \text{weakly in} \quad W_0^{2,k}(\Omega),
$$
then
$$
S_k[u_n] \overset{*}{\rightharpoonup} S[u] \qquad \text{weakly$-*$ in} \quad h^1_r(\Omega),
$$
and
$$
S_k[u_n] \overset{*}{\rightharpoonup} S[u] \qquad \text{weakly$-*$ in} \quad \mathcal{M}(\Omega).
$$
\end{theorem}

\begin{proof}
With respect to weak continuity in $h^1_r(\Omega)$, the case $k=N/2$ was already proven in theorem~\ref{weakcontast}, and the proof for
the rest of cases follows identically by using the result of lemma~\ref{hardyz2}. Weak continuity in $\mathcal{M}(\Omega)$ follows in all cases from
$C_0(\Omega) \subset \mathrm{vmo}_z(\bar{\Omega})$.
\end{proof}

\subsection{Several $k-$Hessians and the Laplacian}

Note that equations~\eqref{rkhessian} are always posed for $\alpha \ge 2$ within this work. This is a necessary condition in order to have a semilinear problem;
the option $\alpha =1$ always leads to a fully nonlinear problem. In fact, there is a way of connecting the different $k-$Hessians with linear
combinations of $k-$Hessians as well as the Laplacian, which is nothing but the $1-$Hessian, making even more explicit the fully nonlinear character
of any equation that contains this type of nonlinearity and that is harmonic rather than polyharmonic.

Let us make this fact more explicit. As already noted in~\cite{n5}\footnote{This was actually a personal advise of Neil Trudinger.},
in $N=2$ the following equality holds
$$
S_2\left[ u+\frac{x_1^2 + x_2^2}{2} \right]= S_2[u] + \Delta u + 1.
$$
This can be generalized for arbitrary $k$ and $N$ in the following way\footnote{Note that this equality is related to the following formula
for the characteristic polynomial of a square matrix $A$ of order $N$
$$
\det(A-\mu I) = \sum_{i=0}^N \mathrm{Tr}_i(A)(-\mu)^{N-i},
$$
where $\mathrm{Tr}_i(A)$, $1 \le i \le N$, is the $i^{\mathrm{th}}$ elementary symmetric polynomial of eigenvalues of matrix $A$ and $\mathrm{Tr}_0(A):=1$.
This classical result can be generalized in the following way
$$
\mathrm{Tr}_k(A-\mu I)= \sum_{i=0}^k \left(\!\!\! \begin{array}{c} N-i \\ k-i \end{array} \!\!\!\right) \mathrm{Tr}_i(A)(-\mu)^{k-i},
$$
which is the linear algebraic counterpart of equation~\eqref{lckhessians} in the main text.
}
\begin{equation}\label{lckhessians}
S_k\left[ u+\frac{x_1^2+ \cdots + x_N^2}{2} \right]= \sum_{i=0}^k \left(\!\!\! \begin{array}{c} N-i \\ k-i \end{array} \!\!\!\right) S_i[u],
\end{equation}
where $\left(\!\!\! \begin{array}{c} N-i \\ k-i \end{array} \!\!\!\right)= \frac{(N-i)!}{(N-k)!(k-i)!}$, $S_1[u]= \Delta u$ and $S_0[u]:=1$.
Despite the potential interest of this change of variables in order to reduce equations, either fully nonlinear or semilinear, with several $k-$Hessians to
equations with a lower number of nonlinearities, we will conclude this section with this brief note and leave any research in this respect for the future.

\vskip5mm
\noindent
{\footnotesize
Carlos Escudero\par\noindent
Departamento de Matem\'aticas\par\noindent
Universidad Aut\'onoma de Madrid\par\noindent
{\tt carlos.escudero@uam.es}\par\vskip1mm\noindent
\& \par\vskip1mm\noindent
Instituto de Ciencias Matem\'aticas\par\noindent
Consejo Superior de Investigaciones Cient\'{\i}ficas\par\noindent
{\tt cel@icmat.es}\par\vskip1mm\noindent
}

\begin{thebibliography}{99}

\bibitem{ar} A. Ambrosetti and P. H. Rabinowitz, {\it Dual variational methods in critical point theory and applications},
J. Functional Analysis {\bf 14} (1973) 349--381.

\bibitem{AGGM} G. Arioli, F. Gazzola, H.-C. Grunau and E. Mitidieri, \textit{A semilinear fourth order elliptic problem
with exponential nonlinearity}, SIAM J. Math. Anal. {\bf 36} (2005) 1226--1258.

\bibitem{ekeland} J. P. Aubin and I. Ekeland, {\it  Applied Nonlinear Analysis}. Ed. John Wiley, 1984.

\bibitem{ball} J. M. Ball, {\it Convexity conditions and existence theorems in nonlinear elasticity},
Arch. Rational Mech. Anal. {\bf 63} (1977) 337--403.

\bibitem{BG} E. Berchio and F. Gazzola, \textit{Some remarks on biharmonic elliptic problems with positive, increasing
and convex nonlinearities}, Electronic J. Differ. Equ. {\bf 34} (2005) 1--20.

\bibitem{brezis1} H. Brezis, N. Fusco and C. Sbordone, \textit{Integrability for the Jacobian of
orientation preserving mappings}, J. Funct. Anal. {\bf 115} (1993) 425--431.

\bibitem{brezis2} H. Brezis and H.-M. Nguyen, \emph{On the distributional Jacobian of maps from $\mathbb{S}^N$ into $\mathbb{S}^N$ in
fractional Sobolev and H\"{o}lder spaces}, Annals of Mathematics {\bf 173} (2011) 1141--1183.

\bibitem{brezis3} H. Brezis and H.-M. Nguyen, \emph{The Jacobian determinant revisited}, Inventiones Mathematicae {\bf 185} (2011) 17--54.

\bibitem{brezis4} H. Brezis and L. Nirenberg, \textit{Degree theory and BMO: I}, Sel. Math. {\bf 2} (1995) 197--263.

\bibitem{brezis5} H. Brezis and L. Nirenberg, \textit{Degree theory and BMO: II}, Sel. Math. {\bf 3} (1996) 309--368.

\bibitem{caffarelli1} L. A. Caffarelli, \textit{Interior $W^{2,p}$ estimates for solutions of Monge-Amp\`ere equations},
Ann. Math. {\bf 131} (1990) 135--150.

\bibitem{caffarelli2} L. A. Caffarelli, L. Nirenberg and J. Spruck, \textit{Dirichlet problem for nonlinear second
order elliptic equations I, Monge-Amp\`ere equations}, Comm. Pure Appl. Math. {\bf 37} (1984) 369--402.

\bibitem{caffarelli3} L. A. Caffarelli, L. Nirenberg and J. Spruck, \textit{Dirichlet problem for nonlinear second order
elliptic equations III, Functions of the eigenvalues of the Hessian}, Acta Math. {\bf 155} (1985) 261--301.

\bibitem{der} D.-C. Chang, G. Dafni and E. M. Stein, {\it Hardy spaces, BMO, and boundary value problems for the Laplacian on a
smooth domain in $\mathbb{R}^N$}, Transactions of the Amer. Math. Soc. {\bf 351} (1999) 1605--1661.

\bibitem{wang1} K. S. Chou and X.-J. Wang, \emph{Variational theory for Hessian equations}, Comm. Pure
Appl. Math. {\bf 54} (2001) 1029--1064.

\bibitem{grafakos} R. R. Coifman and L. Grafakos, \textit{Hardy space estimates for multilinear operators, I},
Revista Matem\'{a}tica Iberoamericana {\bf 8} (1992) 45--67.

\bibitem{coifman} R. Coifman, P. L. Lions, Y. Meyer, and S. Semmes, {\it Compensated compactness and Hardy spaces},
J. Math. Pures Appl. {\bf 72} (1993) 247--286.

\bibitem{CEGM} C. Cowan, P. Esposito, N. Ghoussoub and A. Moradifam, \textit{The critical dimension for a fourth order elliptic problem
with singular nonlinearity}, Archive for Rational Mechanics and Analysis {\bf 198} (2010) 763--787.

\bibitem{dacorogna} B. Dacorogna and F. Murat, \textit{On the optimality of certain Sobolev exponents for the weak continuity of determinants},
J. Funct. Anal. {\bf 105} (1992) 42--62.

\bibitem{DDGM} J. D\'{a}vila, L. Dupaigne, I. Guerra and M. Montenegro, \textit{Stable solutions for the bilaplacian
with exponential nonlinearity}, SIAM J. Math. Anal. {\bf 39} (2007) 565--592.

\bibitem{DFG} J. D\'{a}vila, I. Flores and I. Guerra, \textit{Multiplicity of solutions for a fourth order equation with
power-type nonlinearity}, Math. Ann. {\bf 348} (2009) 143--193.

\bibitem{ekeland2} I. Ekeland, {\it On the variational principle}, J. Math. Anal. Appl. {\bf 47} (1974) 324--353.

\bibitem{escudero} C. Escudero, {\it Geometric principles of surface growth}, Phys. Rev. Lett. {\bf 101} (2008) 196102.

\bibitem{n1} C. Escudero, F. Gazzola, R. Hakl, I. Peral and P.~J. Torres, {\it Existence results for a fourth order partial differential equation arising in
condensed matter physics}, Mathematica Bohemica, in press.

\bibitem{n2} C. Escudero, F. Gazzola and I. Peral, {\it Global existence versus blow-up results
for a fourth order parabolic PDE involving  the Hessian}, J. Math. Pures Appl. {\bf 103} (2015) 924--957.

\bibitem{n3} C. Escudero, R. Hakl, I. Peral and P.~J. Torres, {\it On radial stationary solutions to a model of nonequilibrium growth},
Eur. J. Appl. Math. {\bf 24} (2013) 437--453.

\bibitem{n4} C. Escudero, R. Hakl, I. Peral and P.~J. Torres, {\it Existence and nonexistence results for a singular boundary value problem arising
in the theory of epitaxial growth}, Math. Methods Appl. Sci. {\bf 37} (2014) 793--807.

\bibitem{escudero2} C. Escudero and E. Korutcheva, {\it Origins of scaling relations in nonequilibrium growth}, J. Phys. A: Math. Theor. {\bf 45} (2012) 125005.

\bibitem{n5} C. Escudero and I. Peral, \textit{Some fourth order nonlinear elliptic problems related to epitaxial growth}, J. Differential Equations {\bf 254}
(2013) 2515--2531.

\bibitem{n6} C. Escudero and P. J. Torres, \textit{Existence of radial solutions to biharmonic k-Hessian equations},
Journal of Differential Equations, in press.

\bibitem{FeVe} F. Ferrari and I. E. Verbitsky, \textit{Radial fractional Laplace operators and Hessian inequalities},
J. Differential Equations {\bf 253} (2012) 244-–272.

\bibitem{FG} A. Ferrero and H.-C. Grunau, \textit{The Dirichlet problem for supercritical biharmonic equations
with powertype nonlinearity}, J. Differ. Equ. {\bf 234} (2007) 582--606.

\bibitem{FGK} A. Ferrero, H.-C. Grunau and P. Karageorgis, \textit{Supercritical biharmonic equations with powertype
nonlinearity}, Annali di Matematica {\bf 188} (2009) 171--185.

\bibitem{GAP} J. Garc\'{\i}a Azorero and I. Peral, \textit{Multiplicity of solutions for elliptic problems with critical exponents or with a non-symmetric
term}, Trans. Amer. Math. Soc. {\bf 323} (1991) 877-–895.

\bibitem{GGS} F. Gazzola, H. Grunau and G. Sweers, {\it Polyharmonic boundary value problems. Positivity preserving and nonlinear
higher order elliptic equations in bounded domains}, Lecture Notes in Mathematics, 1991. Springer-Verlag, Berlin, 2010.

\bibitem{giaquinta} M. Giaquinta, G. Modica and J. Sou\v{c}ek, \textit{Cartesian currents in the calculus
of variations I and II}, Ergebnisse der Mathematik und Ihrer Grenzgebiete, vol. 38, Springer-Verlag, Berlin, 1998.

\bibitem{gilbarg} D. Gilbarg and N. S. Trudinger, \textit{Elliptic Partial Differential Equations of Second
Order}, Springer, 1983.

\bibitem{goldberg} D. Goldberg, \textit{A local version of real Hardy spaces}, Duke Math. J. {\bf 46} (1979) 27--42.

\bibitem{grafakos1} L. Grafakos, \textit{Hardy space estimates for multilinear operators, II}, Revista Matem\'{a}tica Iberoamericana {\bf 8} (1992) 69--92.

\bibitem{hajlasz} P. Haj{\l}asz, \textit{Note on weak approximation of minors},
Ann. Inst. H. Poincar\'e {\bf 12} (1995) 415--424.

\bibitem{wang2} N. M. Ivochkina, N. S. Trudinger and X.-J. Wang, \textit{The Dirichlet problem for degenerate
Hessian equations}, Comm. Partial Diff. Eqns {\bf 29} (2004) 219--235.

\bibitem{iwaniec1} T. Iwaniec and G. Martin, \textit{Geometric Function Theory and Nonlinear Analysis}, Oxford
Mathematical Monographs, Oxford University Press, Oxford, 2001.

\bibitem{iwaniec2} T. Iwaniec and J. Onninen, \textit{$\mathcal{H}^1-$Estimates of Jacobians by subdeterminants},
Mathematische Annalen {\bf 324} (2002) 341--358.

\bibitem{iwaniec3} T. Iwaniec and C. Sbordone, \textit{On the integrability of the Jacobian under
minimal hypotheses}, Arch. Rational Mech. Anal. {\bf 119} (1992) 129--143.

\bibitem{labutin} D. Labutin, \textit{Potential estimates for a class of fully nonlinear elliptic equations},
Duke Math. J. {\bf 111} (2002) 1--49.

\bibitem{moradifam} A. Moradifam, \textit{The singular extremal solutions of the bi-laplacian with exponential nonlinearity},
Proc. Amer. Math. Soc. {\bf 138} (2010) 1287--1293.

\bibitem{morrey} C. B. Morrey, \textit{Multiple Integrals in the Calculus of Variations}, Springer-Verlag, Berlin, 1966.

\bibitem{muller1} S. M\"uller, \textit{Weak continuity of determinants and nonlinear elasticity}, C. R. Acad. Sci. Paris {\bf 307} (1988) 501--506.

\bibitem{muller2} S. M\"uller, \textit{Det = det. A Remark on the distributional determinant}, C. R. Acad. Sci. Paris {\bf 311} (1990) 13--17.

\bibitem{muller3} S. M\"uller, \textit{Higher integrability of determinants and weak convergence in $L^1$}, J. Reine Angew. Math. {\bf 412} (1990) 20--34.

\bibitem{muller4} S. M\"uller, \textit{On the singular support of the distributional determinant},
Annales Institut Henri Poincar\'e, Analyse Non Lin\'eaire {\bf 10} (1993) 657--696.

\bibitem{muller5} S. M\"uller, Q. Tang and S. B. Yan, \textit{On a new class of elastic deformations not allowing for cavitation},
Annales Institut Henri Poincar\'e, Analyse Non Lin\'eaire {\bf 11} (1994) 217--243.

\bibitem{reshetnyak} Y. Reshetnyak, \textit{Weak convergence and completely additive vector functions on a set}, Sibir. Math. {\bf 9} (1968) 1039--1045.

\bibitem{wang3} W. M. Sheng, N. S. Trudinger and X.-J. Wang, \textit{The Yamabe problem for higher order
curvatures}, J. Diff. Geom. {\bf 77} (2007) 515--553.

\bibitem{steinb} E. M. Stein, {\it Harmonic Analysis: Real-Variable Methods, Orthogonality and Oscillatory Integrals}, Princeton
University Press, Princeton, New Jersey, 1993.

\bibitem{stein}  E. M. Stein and G. Weiss, {\it On the theory of harmonic functions of several variables. I. The theory of
$H^{p}$-spaces}, Acta Math. {\bf 103} (1960) 25--62.

\bibitem{trudinger} N. S. Trudinger, \textit{On the Dirichlet problem for Hessian equations},
Acta Math. {\bf 175} (1995) 151--164.

\bibitem{trudinger1} N. S. Trudinger, \textit{Weak solutions of Hessian equations}, Comm. Partial Differential
Equations {\bf 22} (1997) 1251--1261.

\bibitem{wang4} N. S. Trudinger and X.-J. Wang, \textit{Hessian measures I}, Topol. Methods Nonlin.
Anal. {\bf 10} (1997) 225--239.

\bibitem{wang5} N. S. Trudinger and X.-J. Wang, \textit{Hessian measures II}, Ann. Math. {\bf 150}
(1999) 579--604.

\bibitem{wang6} N. S. Trudinger and X.-J. Wang, \textit{Hessian measures III}, J. Funct. Anal. {\bf 193}
(2002) 1--23.

\bibitem{wang7} N. S. Trudinger and X.-J. Wang, \textit{A Poincar\'{e} type inequality for Hessian integrals},
Calc. Var. and PDE {\bf 6} (1998) 315--328.

\bibitem{wang8} N. S. Trudinger and X.-J. Wang, \textit{The weak continuity of elliptic operators and
applications in potential theory}, Amer. J. Math. {\bf 551} (2002) 11--32.

\bibitem{wang9} N. S. Trudinger and X.-J. Wang, \textit{Boundary regularity for the Monge-Amp\`{e}re and
affine maximal surface equations}, Ann. Math. {\bf 167} (2008) 993--1028.

\bibitem{wang10} X.-J. Wang, \textit{Existence of multiple solutions to the equations of Monge-Amp\`{e}re
type}, J. Diff. Eqns {\bf 100} (1992) 95--118.

\bibitem{wang11} X.-J. Wang, \textit{A class of fully nonlinear elliptic equations and related functionals},
Indiana Univ. Math. J. {\bf 43} (1994) 25--54.

\bibitem{wang12} X.-J. Wang, \textit{Some counterexamples to the regularity of Monge-Amp\`{e}re equations},
Proc. Amer. Math. Soc. {\bf 123} (1995) 841--845.

\bibitem{wang} X.-J. Wang, \emph{The $k-$Hessian equation}, Lectures Notes in Mathematics {\bf 1977} (2009) 177--252.

\end{thebibliography}
\end{document}